\documentclass{amsart}

\usepackage{amssymb,enumitem,mathtools}
\setlist{itemsep=4pt,topsep=0pt,leftmargin=17pt,listparindent=11pt}
\usepackage{xcolor}
\definecolor{BIT}{cmyk}{1,0,1,0}
\usepackage[colorlinks,citecolor=BIT]{hyperref}
\usepackage{aliascnt}
\usepackage[capitalize]{cleveref}
\usepackage{booktabs,needspace,tikz,capt-of}
\usetikzlibrary{calc,positioning,arrows.meta,decorations.pathreplacing}
\numberwithin{figure}{section}
\newtheorem{theorem}{Theorem}[section]
\newtheorem{conjecture}[theorem]{Conjecture}
\newtheorem{lemma}[theorem]{Lemma}
\newtheorem{proposition}[theorem]{Proposition}
\newtheorem{corollary}[theorem]{Corollary}
\theoremstyle{definition}

\theoremstyle{remark}

\AddToHook{env/conjecture/begin}{\crefalias{theorem}{conjecture}}
\AddToHook{env/lemma/begin}{\crefalias{theorem}{lemma}}
\AddToHook{env/proposition/begin}{\crefalias{theorem}{proposition}}
\AddToHook{env/corollary/begin}{\crefalias{theorem}{corollary}}
\AddToHook{env/definition/begin}{\crefalias{theorem}{definition}}
\AddToHook{env/remark/begin}{\crefalias{theorem}{remark}}
\numberwithin{equation}{section}
\numberwithin{table}{section}
\allowdisplaybreaks
\usepackage[numbers,sort&compress,nonamebreak,merge,elide,longnamesfirst]{natbib}
\usepackage{etoolbox,hyphenat}
\makeatletter
\patchcmd{\NAT@citexnum}
{\@citea \NAT@test{\@ne}\NAT@spacechar\NAT@mbox{\NAT@super@kern\NAT@@open}}
{\@citea \nohyphens{\NAT@test{\@ne}}\nobreakspace{}\NAT@mbox{\NAT@super@kern\NAT@@open}}
{}{}
\makeatother
\BeforeBeginEnvironment{definition}{\Needspace{12\baselineskip}}

\DeclareMathOperator{\inc}{inc}

\DeclareMathOperator{\rk}{rk}

\newcommand{\ind}[1]{\mathbf1_{\{#1\}}}
\newcommand{\bchain}[1]{\mathbf{#1}}

\title[Schur positivity of the lattices \(\mathbf m\times\mathbf n\)]
{Schur positivity of the lattices \(\mathbf m\times\mathbf n\)}
\author[K.~Zhang]{Kai Zhang}
\address[Kai Zhang]{School of Mathematics and Statistics, Beijing Institute of Technology, Beijing 102400, P.\ R.\ China.}
\email{kai@bit.edu.cn}
\date{October 6, 2026}
\hypersetup{
  pdftitle={Schur positivity of the lattices m x n},
  pdfauthor={Kai Zhang},
  pdfsubject={Schur positivity of the lattices m x n}
}

\begin{document}
\begin{abstract}
We prove that the product of chains $\bchain m\times\bchain n$ is not Schur positive whenever $n\ge4$ and $m\ge3n-1$. Writing $m=n+k$, we exhibit a negative Schur coefficient indexed by
$(2n+k-1,2n+k-3,\ldots,k+5,k,k,4)$
and evaluate it explicitly as $(n-2)!$ times a polynomial of degree three in $k$. Together with earlier results, this establishes the two-case
interpretation of the conjecture of Li, Qiu, Yang, and Zhang:
these products are not Schur positive when $n=2$ and $m\ge8$,
or when $n\ge3$ and $m\ge n+5$.
\end{abstract}
\maketitle

\section{Introduction}\label{sec:intro}

The chromatic polynomial, introduced by \citet{Bir12} in connection with the Four Color Problem, counts the proper colorings of a graph with a prescribed number of available colors. \citet{Sta95} extended this invariant to the \emph{chromatic symmetric function}. For a finite simple graph $G$, it is defined by
\[
X_G=\sum_{\kappa\colon V(G)\to\{1,2,\ldots\}}
x_1^{|\kappa^{-1}(1)|}x_2^{|\kappa^{-1}(2)|}\cdots,
\]
where the sum is over proper colorings: adjacent vertices receive distinct colors. The exponents record the sizes of the color classes. Since permuting the colors permutes the summands, $X_G$ is a symmetric function, homogeneous of degree $|V(G)|$.

Let $\Lambda$ denote the ring of symmetric functions. Given a basis $\{b_\lambda\}$ of $\Lambda$, a symmetric function is \emph{$b$-positive} if its expansion in this basis has only nonnegative coefficients. In particular, the Schur basis $\{s_\lambda\}$ and the elementary basis $\{e_\lambda\}$ give the notions of \emph{Schur positivity} and \emph{$e$-positivity}; see \citet{Mac95,Sta99}. Schur functions are the Frobenius characteristics of the irreducible representations of symmetric groups. Thus a homogeneous symmetric function with nonnegative integral Schur coefficients is the Frobenius characteristic of a representation, and its Schur coefficients record the irreducible multiplicities; see \citet{Sag01,Sta99}. Moreover, every $e$-positive symmetric function is Schur positive, since elementary symmetric functions are Schur functions indexed by columns and products of Schur functions are Schur positive. A graph $G$ is called $b$-positive when $X_G$ is $b$-positive.

The incomparability graph of a finite poset $P$, denoted by $\inc(P)$, joins incomparable elements. \citet{Sta95} reformulated a conjecture of \citet{SS93} as the assertion that $X_{\inc(P)}$ is $e$-positive whenever $P$ is $(3+1)$-free, that is, it contains no induced subposet consisting of a three-element chain and an element incomparable with all three. \citet{Gas96} proved the weaker assertion of Schur positivity. \citet{GP13} reduced the $e$-positivity conjecture to unit interval orders. \citet{Hik25} later proved the Stanley--Stembridge conjecture. The $e$-positivity result applies to complete graphs, paths, and lollipop graphs. The $e$-positivity of lollipop graphs was previously established by \citet{Dv18}.

The incomparability graph of every $(3+1)$-free poset is claw-free. The broader Stanley--Gasharov conjecture, recorded by \citet{Sta98} and studied by \citet{Gas99}, asserted that every claw-free graph is Schur positive. Here a claw is the graph $K_{1,3}$, and claw-free means that no induced subgraph is a claw. \citet{Sv25} proved Schur positivity for generalized nets. The conjecture was later disproved independently by \citet{Pra26} and \citet{MM26}, and \citet{WZZ26} constructed infinite families of counterexamples.

Other graph families admit explicit Schur-positivity classifications. \citet{WW20} determined the Schur-positive complete bipartite and tripartite graphs, and \citet{Sv26} extended the classification to complete multipartite graphs. A complete multipartite graph is the incomparability graph of a disjoint union of chains, whose lengths are its part sizes.

Several methods for studying chromatic symmetric functions have developed alongside these results. \citet{GS01} introduced a chromatic symmetric function in noncommuting variables. \citet{SW16} defined chromatic quasisymmetric functions, which retain an additional coloring statistic, and \citet{Paw22} obtained Schur-positivity criteria using the group algebra of the symmetric group. To prove non-Schur-positivity, we use the formula of \citet{WW20}, based on the inverse-Kostka interpretation of \citet{ER90}. The formula expresses each Schur coefficient as a signed sum of semiordered stable-partition counts; see \cref{thm:WW}. For an incomparability graph, its stable sets are the chains of the poset, so the formula becomes a signed enumeration of chain partitions.

Necessary conditions for Schur positivity can also be formulated in terms of stable partitions. A stable partition divides the vertex set into sets containing no adjacent vertices, and its type lists the block sizes in weakly decreasing order. Following \citet{Sta98}, a graph is \emph{nice} if its stable-partition types are closed downward in dominance order. Informally, the nice property requires a stable partition to exist for every more balanced type allowed by dominance. The \emph{strongly nice} property of \citet{LLYZ24} strengthens this requirement: the coefficients in the monomial basis $\{m_\lambda\}$ must not decrease as the index decreases in dominance order. The precise definitions and the dominance inequalities are recalled in Section~\ref{sec:preliminaries}. \citet{Sta98} proved that every Schur-positive graph is nice, and \citet{LLYZ24} strengthened this conclusion to the strongly nice property. Thus
\[
\text{Schur positive}\ \Longrightarrow\ \text{strongly nice}
\ \Longrightarrow\ \text{nice}.
\]
Neither converse holds in general: \citet{Zha26} constructed infinite families separating these properties.

These notions extend to finite posets through their incomparability graphs: a poset is Schur positive, strongly nice, or nice when its incomparability graph has the corresponding property. Since stable sets in $\inc(P)$ are chains in $P$, the nice property of a poset concerns the possible sizes of the chains in its chain partitions. This interpretation connects Schur positivity with questions about chain decompositions that predate the chromatic symmetric function.

For distributive lattices, the connection originates in a conjecture of \citet{Gri88} about chain partitions of Boolean algebras. Let $B_r$ be the poset of subsets of an $r$-element set ordered by inclusion. Griggs's conjecture is equivalent to the assertion that $B_r$ is nice. Motivated by this problem and the implication from Schur positivity to niceness, \citet[p.~270]{Sta98} asked whether $B_r$ might be Schur positive. Since Boolean algebras are distributive lattices, he also raised the broader question of whether every finite distributive lattice is Schur positive. We recall the order-theoretic definitions and Birkhoff's representation theorem in Section~\ref{sec:preliminaries}.

\citet{LQYZ25} answered the broader question negatively. They constructed distributive lattices that are not nice, and hence not Schur positive, as well as distributive lattices that are nice but not Schur positive. Their examples do not settle the separate question for Boolean algebras. To establish non-Schur-positivity for the nice examples, they used the Wang--Wang formula to obtain an explicit negative Schur coefficient.

Let $\bchain n$ denote a chain with $n$ elements. The products $\bchain m\times\bchain n$ are distributive lattices and are always nice by the chain-partition results of \citet{LE99}. Thus their non-Schur-positivity cannot be deduced from a failure of niceness. These products include the family of nice counterexamples studied by \citet{LQYZ25}, and they are the subject of this paper. Since interchanging the factors gives an isomorphic poset, we assume $m\ge n$. Li, Qiu, Yang, and Zhang proved the following result.

\begin{theorem}[Li, Qiu, Yang, and Zhang {\cite[Theorem~4.1]{LQYZ25}}]\label{thm:LQYZ}
For integers $k\ge5$ and $n\ge(k+2)/2$, the product
$\bchain{(n+k)}\times\bchain n$ is not Schur positive.
\end{theorem}
Writing $m=n+k$, their parameter range is equivalently
\begin{equation}\label{eq:old-range}
n\ge4,\qquad n+5\le m\le3n-2.
\end{equation}
They prove the theorem by exhibiting a negative Schur coefficient indexed by
\[
(2n+k-1,2n+k-3,\ldots,k+3,k-3,2,2).
\]
For this shape, the rank bound fixes the $n-1$ largest parts of every chain-partition type contributing to the Wang--Wang formula. After choosing chains of these lengths, they count the partitions of the remaining elements into two or three chains. Their theorem and further SageMath computations motivated the following conjecture, whose printed parameter conditions we reproduce.

\begin{conjecture}[Li, Qiu, Yang, and Zhang {\cite[Conjecture~5.1]{LQYZ25}}]\label{conj:LQYZ}
The product $\bchain m\times\bchain n$ is not Schur positive whenever
\[
(n\ge2\text{ and }m\ge8)
\qquad\text{or}\qquad
(n\ge3\text{ and }m\ge n+5).
\]
\end{conjecture}

In this paper we address the following two-case interpretation of the conjectured range:\footnote{The printed first condition contains the second and includes $m=n\ge8$, whereas \citet[Conjecture~5.2]{LQYZ25} predict Schur positivity for equal-factor products. Restricting the first condition to $n=2$ is consistent with the two-case formulation and the reported computations; see also \citet{WZ26}. Our conclusion concerns \eqref{eq:conjectured-range}, not the broader literal first condition.}
\begin{equation}\label{eq:conjectured-range}
 n=2,\ m\ge8\qquad\text{or}\qquad n\ge3,\ m\ge n+5.
\end{equation}
\Cref{thm:LQYZ} does not cover $n=2,3$ or $m\ge3n-1$. The computations of \citet{LQYZ25} give evidence beyond the theorem: they found non-Schur-positivity for $8\le m\le1199$ when $n=2$, for $8\le m\le100$ when $n=3$, and for $9\le m\le50$ when $n=4$.

\citet[Theorems~3.4 and~4.8]{WZ26} determined the exact thresholds for $n=2,3$. The product $\bchain m\times\bchain2$ is Schur positive exactly for $m\le7$, and $\bchain m\times\bchain3$ exactly for $m\le6$. For $m\ge8$, they obtain negative coefficients indexed, respectively, by
\[
(m-2,m-2,4)\qquad\text{and}\qquad(m+2,m-3,m-3,4).
\]
Their proofs use Pieri's rules and stable-composition counts to evaluate these coefficients uniformly in $m$. The case $7\times3$ is treated separately, improving the width-three threshold in \eqref{eq:conjectured-range} by one. Together with \cref{thm:LQYZ}, their results leave the range $n\ge4$, $m\ge3n-1$ in \eqref{eq:conjectured-range} to be established.

We prove non-Schur-positivity throughout this remaining range. More precisely, write $m=n+k$ and let
\[
\rho=(2n+k-1,2n+k-3,\ldots,k+5,k,k,4).
\]
The initial segment has $n-2$ parts, and $\rho$ is a partition of $n(n+k)$ of length $n+1$. Our main result gives
\[
[s_\rho]X_{\inc(\bchain{(n+k)}\times\bchain n)}
=\frac{(n-2)!}{360}P(n,k)<0
\qquad(n\ge4,\ k\ge2n-1),
\]
where $P(n,k)$ is a polynomial of degree three in $k$, given explicitly in \cref{thm:main}. Combining \cref{thm:main} with \cref{thm:LQYZ} and the results of \citet{WZ26} for $n=2,3$ proves non-Schur-positivity throughout \eqref{eq:conjectured-range}; see \cref{cor:range}.

The choice of $\rho$ extends the two coefficient indices used by \citet{WZ26}: the same pattern gives $(k,k,4)$ for $n=2$ and $(k+5,k,k,4)$ for $n=3$. For $n\ge4$, the rank bound and the special rim hook formula force the first $n-2$ chain lengths. Only six special rim hook tabloids can contribute, and the coefficient is a signed sum of the associated chain-partition counts. In contrast with the choice in \cref{thm:LQYZ}, one fewer long chain is fixed. The complements therefore contain two elements at each internal rank, and their partitions into three chains require a different enumeration.

We first parametrize the possible families of long chains and characterize their complements. The parametrization separates a permutation of the $n-2$ long chains from the data determining their complement, which accounts for the factor $(n-2)!$. In any contributing three-chain partition, the two larger chains have fixed row coordinates throughout a middle interval. The portions before and after this interval have bounded length, independently of $n$ and $k$, because the third chain has at most four elements. Counting these two boundary regions and the middle configurations gives formulas valid uniformly in both parameters. We then combine the six counts and prove that the resulting polynomial is negative for $n\ge4$ and $k\ge2n-1$.

Section~\ref{sec:preliminaries} recalls the notation for partitions, graphs, and posets, together with the rank bound and the Wang--Wang coefficient formula. Section~\ref{sec:main} determines the contributing tabloids, enumerates the corresponding chain partitions, and proves the main theorem. The detailed polynomial evaluations are given in Appendix~\ref{app:counts}.

\section{Preliminaries}\label{sec:preliminaries}

\subsection{Partitions and special rim hook tabloids}

We use the notation of \citet{LQYZ25}. For a positive integer $r$, let $[r]=\{1,\ldots,r\}$. A partition $\lambda=(\lambda_1,\ldots,\lambda_\ell)$ of $N$, written $\lambda\vdash N$, has positive, weakly decreasing parts with sum $N$. Put
\[
\lambda_{\langle j\rangle}=(\lambda_j,\ldots,\lambda_\ell).
\]
For partitions $\lambda$ and $\mu$ of the same integer, we write $\lambda\succeq\mu$ if $\lambda$ \emph{dominates} $\mu$, that is,
\[
\sum_{i=1}^j\lambda_i\ge\sum_{i=1}^j\mu_i\qquad(j\ge1),
\]
where missing parts are interpreted as zero. We use $m_\lambda$, $e_\lambda$, and $s_\lambda$ for the monomial, elementary, and Schur symmetric functions, respectively, and $[b_\lambda]f$ for the coefficient of $b_\lambda$ in the corresponding basis expansion of $f$; see \citet{Mac95,Sta99}.
We also write $\lambda=\langle1^{\alpha_1},2^{\alpha_2},\ldots\rangle$, where $\alpha_r$ is the number of parts of size $r$; zero multiplicities are omitted.

The \emph{Young diagram}, also called the \emph{Ferrers diagram}, of a
partition $\lambda$ consists of $\lambda_i$ left-justified cells in row $i$.
We use the French convention: rows are counted from bottom to top, so the
longest row is at the bottom. For example, the diagram of
$\lambda=(4,4,2,2,1)$ is shown in \cref{fig:prelim-ribbons}(a).
If $\mu_i\le\lambda_i$ for every $i$, with missing parts interpreted as
zero, the \emph{skew diagram} $\lambda/\mu$ is obtained by deleting the
diagram of $\mu$ from that of $\lambda$.

A \emph{ribbon} is a nonempty skew diagram that is connected through shared
cell edges and contains no $2\times2$ square. A \emph{rim hook} of a Young
diagram is a ribbon whose removal leaves a Young diagram. A \emph{rim hook
tabloid} of shape $\lambda$ is a decomposition of its diagram into ribbons
that can be removed successively as rim hooks of the remaining diagrams.
The tabloid records the decomposition, not a choice of removal order.
It is \emph{special} if every ribbon meets the first column;
a special rim hook tabloid is also called a \emph{special ribbon tabloid}.
The middle panel of \cref{fig:prelim-ribbons} is not special because its
five-cell ribbon misses the first column; the right-hand panel is special.

\begin{figure}[htbp]
\centering
\begin{tikzpicture}[x=.43cm,y=.43cm,font=\small,
  line cap=round,line join=round]
\foreach \offset in {0,10,20}{
  \begin{scope}[shift={(\offset,0)}]
    \foreach \x in {1,2,3}
      \draw[line width=.25pt] (\x,0)--(\x,2);
    \draw[line width=.25pt] (0,1)--(4,1);
    \draw[line width=.25pt] (1,2)--(1,5);
    \draw[line width=.25pt] (0,2)--(2,2);
    \draw[line width=.25pt] (0,3)--(2,3);
    \draw[line width=.25pt] (0,4)--(1,4);
    \draw[line width=.7pt]
      (0,0)--(4,0)--(4,2)--(2,2)--(2,4)--(1,4)--(1,5)--(0,5)--cycle;
  \end{scope}
}
\node[align=center,anchor=north] at (2,-.35) {(a) Young diagram};
\begin{scope}[shift={(10,0)}]
  \draw[line width=1.1pt] (1.5,3.5)--(1.5,1.5)--(3.5,1.5);
  \draw[line width=1.1pt] (.5,4.5)--(.5,3.5);
  \draw[line width=1.1pt] (.5,2.5)--(.5,1.5);
  \draw[line width=1.1pt] (.5,.5)--(3.5,.5);
  \foreach \x/\y in {.5/.5,3.5/.5,3.5/1.5,.5/2.5,
                       .5/3.5,.5/4.5,.5/1.5,1.5/3.5}
    \fill (\x,\y) circle[radius=1.1pt];
  \node[align=center,anchor=north] at (2,-.35) {(b) Rim hook tabloid\\(not special)};
\end{scope}
\begin{scope}[shift={(20,0)}]
  \draw[line width=1.1pt] (.5,4.5)--(.5,3.5)--(1.5,3.5);
  \draw[line width=1.1pt] (.5,2.5)--(1.5,2.5);
  \draw[line width=1.1pt] (.5,1.5)--(3.5,1.5)--(3.5,.5);
  \draw[line width=1.1pt] (.5,.5)--(2.5,.5);
  \foreach \x/\y in {.5/4.5,1.5/3.5,.5/2.5,1.5/2.5,
                       .5/1.5,3.5/.5,.5/.5,2.5/.5}
    \fill (\x,\y) circle[radius=1.1pt];
  \node[align=center,anchor=north] at (2,-.35) {(c) Special rim hook\\tabloid $T$};
\end{scope}
\end{tikzpicture}
\caption{The Young diagram of $(4,4,2,2,1)$ and two rim hook tabloids
of this shape, in the French convention. Thick paths join the cell centers
within individual ribbons; dots mark their endpoints.
Adapted from \citet{WZ26}.}
\label{fig:prelim-ribbons}
\end{figure}
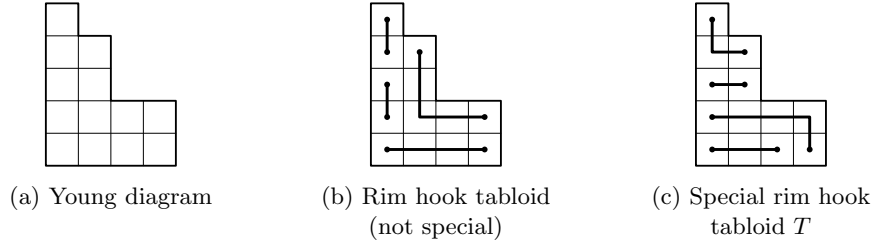

Write $\mathcal T_\lambda$ for the set of special rim hook tabloids of shape
$\lambda$. For $T\in\mathcal T_\lambda$, put
$\operatorname{sh}(T)=\lambda$ and let $\operatorname{cont}(T)$ be the
partition obtained by arranging its ribbon sizes in weakly decreasing
order. The \emph{height} $\operatorname{ht}(B)$ of a ribbon $B$ is one less
than the number of rows it occupies, and
\[
\operatorname{ht}(T)=\sum_{B\in T}\operatorname{ht}(B).
\]
The sign of $T$ is $(-1)^{\operatorname{ht}(T)}$.
For example, the tabloid in \cref{fig:prelim-ribbons}(c) has
$\operatorname{cont}(T)=(5,3,3,2)$ and $\operatorname{ht}(T)=2$, hence sign $+1$.
We have $\operatorname{cont}(T)\succeq\operatorname{sh}(T)$ for every
$T\in\mathcal T_\lambda$; see \citet{ER90}.

\subsection{Graphs, posets, and chain partitions}

A subset of the vertex set of a graph $G$ is \emph{stable} if no two of its vertices are adjacent. A \emph{stable partition} of $G$ is a partition of $V(G)$ into stable sets, and its \emph{type} lists their sizes in weakly decreasing order. Following \citet{Sta98}, the graph $G$ is \emph{nice} if the existence of a stable partition of type $\lambda$ implies the existence of one of type $\mu$ whenever $\lambda\succeq\mu$. Equivalently,
\[
[m_\lambda]X_G>0\quad\Longrightarrow\quad[m_\mu]X_G>0
\qquad(\lambda\succeq\mu).
\]
It is \emph{strongly nice} if
\[
[m_\lambda]X_G\le[m_\mu]X_G\qquad(\lambda\succeq\mu);
\]
see \citet{LLYZ24}.

For a finite poset $P$, the \emph{incomparability graph} $\inc(P)$ has vertex set $P$ and an edge between each pair of incomparable elements. A \emph{chain} is a subset of pairwise comparable elements, and an \emph{antichain} is a subset of pairwise incomparable elements. A chain is \emph{saturated} if consecutive elements are related by a cover relation. The terms Schur positive, strongly nice, and nice for $P$ refer to the corresponding properties of $\inc(P)$. A poset is \emph{$(3+1)$-free} if it contains no induced subposet isomorphic to the disjoint union of a three-element chain and a singleton. A graph is \emph{claw-free} if it contains no induced copy of $K_{1,3}$.

A \emph{lattice} is a poset in which every pair of elements $x,y$ has a least upper bound $x\vee y$ and a greatest lower bound $x\wedge y$. It is \emph{distributive} if meet and join satisfy the distributive laws. By Birkhoff's representation theorem, every finite distributive lattice is isomorphic to the lattice of order ideals of a finite poset, with union and intersection as join and meet; see \citet{DP02}. An order ideal is a subset $I$ such that $y\in I$ and $x\le y$ imply $x\in I$. The Boolean algebra $B_r$ is the lattice of all subsets of $[r]$, ordered by inclusion.

A \emph{semiordered stable partition} of a graph $G$ is a partition of
$V(G)$ into stable sets, ordered by weakly decreasing size, with an order
also specified among sets of equal size. Its \emph{type} is the partition
of its block sizes. For example, for type $(5,3,3,2)$, interchanging the
two blocks of size $3$ gives a different semiordered partition; each
unordered stable partition of this type therefore gives two semiordered
ones.

A \emph{semiordered chain partition} of a finite poset $P$ is defined in the same way, with chains in place of stable sets. Write $\operatorname{SCP}_{P,\lambda}$ for the set of such partitions of type $\lambda$; these correspond to semiordered stable partitions of $\inc(P)$.

The product $P=\bchain m\times\bchain n$ is identified with $[m]\times[n]$, ordered by
\[
(x,y)\le(x',y')\quad\Longleftrightarrow\quad x\le x',\ y\le y'.
\]
It is a distributive lattice, with meet and join given by coordinatewise minimum and maximum. Assume $m\ge n$. Its rank function is $\rk(x,y)=x+y-2$. If $P(q)$ denotes its rank-$q$ set and $a_q=|P(q)|$, then
\begin{equation}\label{eq:ranks}
(a_0,a_1,\ldots,a_{m+n-2})=
\left(1,2,\ldots,n-1,\underbrace{n,\ldots,n}_{m-n+1\text{ ranks}},n-1,\ldots,2,1\right).
\end{equation}
In Hasse diagrams, elements of the same rank are read from left to right in increasing order of $x$ and decreasing order of $y$.

We use the following rank bound to control the long chains. Since a
chain meets each rank at most once, pairwise disjoint chains
$C_1,\ldots,C_r$, with $1\le r\le n$, satisfy
\begin{equation}\label{eq:capacity}
\sum_{i=1}^r|C_i|\le\sum_q\min(r,a_q)=r(m+n-r).
\end{equation}
Suppose that the pairwise disjoint chains $C_1,\ldots,C_t$, with
$1\le t\le n$, satisfy $|C_i|=m+n-2i+1$ for $1\le i\le t$.
For every $1\le r\le t$, the first $r$ chains attain the bound in
\eqref{eq:capacity}; since their contribution at rank $q$ is at most
$\min(r,a_q)$, equality holds at every rank. Subtracting consecutive
prefixes shows that $C_i$ has one vertex at each rank
$i-1,\ldots,m+n-i-1$ and none elsewhere, so it is saturated.
This is the rank argument used in
\citet[Proposition~4.3]{LQYZ25}.

\subsection{The Schur coefficient formula}

The following formula expresses a Schur coefficient as a signed count of
semiordered stable partitions.
\begin{theorem}[Wang and Wang {\cite[Theorem~3.1]{WW20}}]\label{thm:WW}
For a finite simple graph $G$ and $\lambda\vdash|V(G)|$,
\begin{equation}\label{eq:WW-graph}
[s_\lambda]X_G
=\sum_{T\in\mathcal T_\lambda}(-1)^{\operatorname{ht}(T)}N_T,
\end{equation}
where $N_T$ is the number of semiordered stable partitions of $G$ of type $\operatorname{cont}(T)$.
\end{theorem}

Since stable sets of $\inc(P)$ are precisely chains of $P$,
\cref{thm:WW} gives the following poset version.
\begin{corollary}[Li, Qiu, Yang, and Zhang {\cite[Corollary~2.3]{LQYZ25}}]\label{cor:WW-poset}
For a finite poset $P$ and a partition $\lambda\vdash|P|$,
\begin{equation}\label{eq:WW}
[s_\lambda]X_{\inc(P)}
=\sum_{T\in\mathcal T_\lambda}(-1)^{\operatorname{ht}(T)}
\bigl|\operatorname{SCP}_{P,\operatorname{cont}(T)}\bigr|.
\end{equation}
\end{corollary}

\section{A negative Schur coefficient for products of two chains}\label{sec:main}

Throughout this section, let $m=n+k$, where $n\ge4$ and $k\ge2n-1$. We prove the following theorem.

\begin{theorem}\label{thm:main}
The product $\bchain{(n+k)}\times\bchain n$ is not Schur positive for $n\ge4$ and $k\ge2n-1$. More precisely, for
\begin{equation}\label{eq:rho}
\rho=(2n+k-1,2n+k-3,\ldots,k+5,k,k,4),
\end{equation}
we have
\begin{equation}\label{eq:main-coeff}
[s_{\rho}]X_{\inc(\bchain{(n+k)}\times\bchain n)}
=\frac{(n-2)!}{360}P(n,k)<0,
\end{equation}
where
\begin{equation}\label{eq:P}
\begin{aligned}
P(n,k)={}&(-20n^3+20n)k^3+(15n^4+30n^3-195n^2+510n-360)k^2
+(-6n^5-15n^4-100n^3\\
&\quad +195n^2+2446n-3240)k+n^6+3n^5+25n^4+45n^3+4294n^2-15168n+12240.
\end{aligned}
\end{equation}
\end{theorem}

The initial segment of $\rho$ has $n-2$ parts, and its tail is $\rho_{\langle n-1\rangle}=(k,k,4)$. Its parts sum to $n(n+k)$. As in \citet[Lemma~4.2]{LQYZ25}, we first determine which special rim hook tabloids can contribute to \eqref{eq:WW}.

\begin{lemma}\label{lem:force}
Let $T\in\mathcal T_\rho$ have content $\lambda$, and suppose that $\operatorname{SCP}_{\bchain m\times\bchain n,\lambda}\ne\varnothing$. Then
\[
\lambda_i=\rho_i=m+n-2i+1\qquad(1\le i\le n-2),
\]
and the first $n-2$ rows of $T$ are single horizontal rim hooks.
\end{lemma}
\begin{proof}
By $\lambda\succeq\rho$ and \eqref{eq:capacity},
\[
\sum_{j=1}^i\rho_j\le\sum_{j=1}^i\lambda_j
\le i(m+n-i)=\sum_{j=1}^i\rho_j
\qquad(1\le i\le n-2).
\]
Thus the prefix sums are equal, and subtraction gives $\lambda_i=\rho_i$.
For the row assertion, proceed by induction on $i$. Suppose that the first
$i-1$ rows are horizontal hooks. These account for the $i-1$ largest
parts of $\lambda$, so every other hook has size at most $\lambda_i=\rho_i$.
The hook containing the rightmost
cell of row $i$ meets the first column, so connectedness forces it to
have at least $\rho_i$ cells, with equality only when it is precisely
row $i$. This completes the induction.
\end{proof}

By \cref{lem:force}, only the three-row tail $(k,k,4)$ remains.
It has the six special rim hook tabloids shown in \cref{fig:six-ribbons}:
each of the three choices for the top ribbon leaves a two-row shape
with two special rim hook tabloids.
Let $T_j$ be the full tabloid obtained by adjoining the forced horizontal
rows to the $j$th tail, and let $\mu_j$ be its tail content, as labeled
in \cref{fig:six-ribbons}. Since every tail hook has at most $k+2<\rho_{n-2}=k+5$
cells, $\operatorname{cont}(T_j)_{\langle n-1\rangle}=\mu_j$.

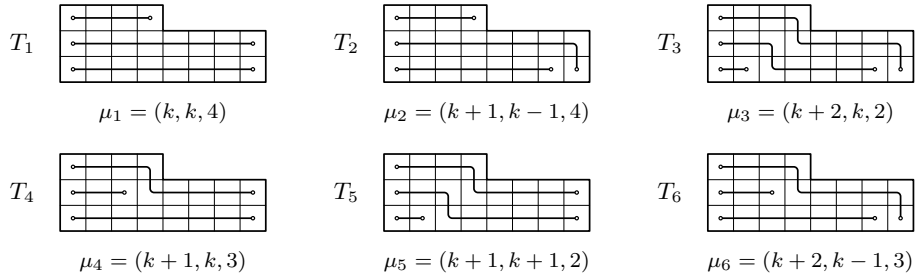
\begin{figure}[htbp]\centering
\begin{tikzpicture}[x=.34cm,y=.34cm,font=\small,
  line cap=round,line join=round]
\begin{scope}[shift={(0,0)}]
\foreach \x in {1,...,7}
  \draw[black,line width=.35pt] (\x,0)--(\x,2);
\foreach \x in {1,2,3}
  \draw[black,line width=.35pt] (\x,2)--(\x,3);
\draw[black,line width=.35pt] (0,1)--(8,1) (0,2)--(4,2);
\draw[black,line width=.65pt] (0,0)--(8,0)--(8,2)--(4,2)--(4,3)--(0,3)--cycle;
\draw[black,line width=.65pt,rounded corners=1.5pt] (0.5,2.5)--(3.5,2.5);
\draw[black,fill=white,line width=.45pt] (0.5,2.5) circle (.065);
\draw[black,fill=white,line width=.45pt] (3.5,2.5) circle (.065);
\draw[black,line width=.65pt,rounded corners=1.5pt] (0.5,1.5)--(7.5,1.5);
\draw[black,fill=white,line width=.45pt] (0.5,1.5) circle (.065);
\draw[black,fill=white,line width=.45pt] (7.5,1.5) circle (.065);
\draw[black,line width=.65pt,rounded corners=1.5pt] (0.5,0.5)--(7.5,0.5);
\draw[black,fill=white,line width=.45pt] (0.5,0.5) circle (.065);
\draw[black,fill=white,line width=.45pt] (7.5,0.5) circle (.065);
\node[anchor=east] at (-.6,1.5) {$T_1$};
\node[anchor=north,font=\footnotesize] at (4,-.4) {$\mu_1=(k,k,4)$};

\end{scope}
\begin{scope}[shift={(12.6,0)}]
\foreach \x in {1,...,7}
  \draw[black,line width=.35pt] (\x,0)--(\x,2);
\foreach \x in {1,2,3}
  \draw[black,line width=.35pt] (\x,2)--(\x,3);
\draw[black,line width=.35pt] (0,1)--(8,1) (0,2)--(4,2);
\draw[black,line width=.65pt] (0,0)--(8,0)--(8,2)--(4,2)--(4,3)--(0,3)--cycle;
\draw[black,line width=.65pt,rounded corners=1.5pt] (0.5,2.5)--(3.5,2.5);
\draw[black,fill=white,line width=.45pt] (0.5,2.5) circle (.065);
\draw[black,fill=white,line width=.45pt] (3.5,2.5) circle (.065);
\draw[black,line width=.65pt,rounded corners=1.5pt] (0.5,1.5)--(7.5,1.5)--(7.5,0.5);
\draw[black,fill=white,line width=.45pt] (0.5,1.5) circle (.065);
\draw[black,fill=white,line width=.45pt] (7.5,0.5) circle (.065);
\draw[black,line width=.65pt,rounded corners=1.5pt] (0.5,0.5)--(6.5,0.5);
\draw[black,fill=white,line width=.45pt] (0.5,0.5) circle (.065);
\draw[black,fill=white,line width=.45pt] (6.5,0.5) circle (.065);
\node[anchor=east] at (-.6,1.5) {$T_2$};
\node[anchor=north,font=\footnotesize] at (4,-.4) {$\mu_2=(k+1,k-1,4)$};

\end{scope}
\begin{scope}[shift={(25.2,0)}]
\foreach \x in {1,...,7}
  \draw[black,line width=.35pt] (\x,0)--(\x,2);
\foreach \x in {1,2,3}
  \draw[black,line width=.35pt] (\x,2)--(\x,3);
\draw[black,line width=.35pt] (0,1)--(8,1) (0,2)--(4,2);
\draw[black,line width=.65pt] (0,0)--(8,0)--(8,2)--(4,2)--(4,3)--(0,3)--cycle;
\draw[black,line width=.65pt,rounded corners=1.5pt] (0.5,2.5)--(3.5,2.5)--(3.5,1.5)--(7.5,1.5)--(7.5,0.5);
\draw[black,fill=white,line width=.45pt] (0.5,2.5) circle (.065);
\draw[black,fill=white,line width=.45pt] (7.5,0.5) circle (.065);
\draw[black,line width=.65pt,rounded corners=1.5pt] (0.5,1.5)--(2.5,1.5)--(2.5,0.5)--(6.5,0.5);
\draw[black,fill=white,line width=.45pt] (0.5,1.5) circle (.065);
\draw[black,fill=white,line width=.45pt] (6.5,0.5) circle (.065);
\draw[black,line width=.65pt,rounded corners=1.5pt] (0.5,0.5)--(1.5,0.5);
\draw[black,fill=white,line width=.45pt] (0.5,0.5) circle (.065);
\draw[black,fill=white,line width=.45pt] (1.5,0.5) circle (.065);
\node[anchor=east] at (-.6,1.5) {$T_3$};
\node[anchor=north,font=\footnotesize] at (4,-.4) {$\mu_3=(k+2,k,2)$};

\end{scope}
\begin{scope}[shift={(0,-5.8)}]
\foreach \x in {1,...,7}
  \draw[black,line width=.35pt] (\x,0)--(\x,2);
\foreach \x in {1,2,3}
  \draw[black,line width=.35pt] (\x,2)--(\x,3);
\draw[black,line width=.35pt] (0,1)--(8,1) (0,2)--(4,2);
\draw[black,line width=.65pt] (0,0)--(8,0)--(8,2)--(4,2)--(4,3)--(0,3)--cycle;
\draw[black,line width=.65pt,rounded corners=1.5pt] (0.5,2.5)--(3.5,2.5)--(3.5,1.5)--(7.5,1.5);
\draw[black,fill=white,line width=.45pt] (0.5,2.5) circle (.065);
\draw[black,fill=white,line width=.45pt] (7.5,1.5) circle (.065);
\draw[black,line width=.65pt,rounded corners=1.5pt] (0.5,1.5)--(2.5,1.5);
\draw[black,fill=white,line width=.45pt] (0.5,1.5) circle (.065);
\draw[black,fill=white,line width=.45pt] (2.5,1.5) circle (.065);
\draw[black,line width=.65pt,rounded corners=1.5pt] (0.5,0.5)--(7.5,0.5);
\draw[black,fill=white,line width=.45pt] (0.5,0.5) circle (.065);
\draw[black,fill=white,line width=.45pt] (7.5,0.5) circle (.065);
\node[anchor=east] at (-.6,1.5) {$T_4$};
\node[anchor=north,font=\footnotesize] at (4,-.4) {$\mu_4=(k+1,k,3)$};

\end{scope}
\begin{scope}[shift={(12.6,-5.8)}]
\foreach \x in {1,...,7}
  \draw[black,line width=.35pt] (\x,0)--(\x,2);
\foreach \x in {1,2,3}
  \draw[black,line width=.35pt] (\x,2)--(\x,3);
\draw[black,line width=.35pt] (0,1)--(8,1) (0,2)--(4,2);
\draw[black,line width=.65pt] (0,0)--(8,0)--(8,2)--(4,2)--(4,3)--(0,3)--cycle;
\draw[black,line width=.65pt,rounded corners=1.5pt] (0.5,2.5)--(3.5,2.5)--(3.5,1.5)--(7.5,1.5);
\draw[black,fill=white,line width=.45pt] (0.5,2.5) circle (.065);
\draw[black,fill=white,line width=.45pt] (7.5,1.5) circle (.065);
\draw[black,line width=.65pt,rounded corners=1.5pt] (0.5,1.5)--(2.5,1.5)--(2.5,0.5)--(7.5,0.5);
\draw[black,fill=white,line width=.45pt] (0.5,1.5) circle (.065);
\draw[black,fill=white,line width=.45pt] (7.5,0.5) circle (.065);
\draw[black,line width=.65pt,rounded corners=1.5pt] (0.5,0.5)--(1.5,0.5);
\draw[black,fill=white,line width=.45pt] (0.5,0.5) circle (.065);
\draw[black,fill=white,line width=.45pt] (1.5,0.5) circle (.065);
\node[anchor=east] at (-.6,1.5) {$T_5$};
\node[anchor=north,font=\footnotesize] at (4,-.4) {$\mu_5=(k+1,k+1,2)$};

\end{scope}
\begin{scope}[shift={(25.2,-5.8)}]
\foreach \x in {1,...,7}
  \draw[black,line width=.35pt] (\x,0)--(\x,2);
\foreach \x in {1,2,3}
  \draw[black,line width=.35pt] (\x,2)--(\x,3);
\draw[black,line width=.35pt] (0,1)--(8,1) (0,2)--(4,2);
\draw[black,line width=.65pt] (0,0)--(8,0)--(8,2)--(4,2)--(4,3)--(0,3)--cycle;
\draw[black,line width=.65pt,rounded corners=1.5pt] (0.5,2.5)--(3.5,2.5)--(3.5,1.5)--(7.5,1.5)--(7.5,0.5);
\draw[black,fill=white,line width=.45pt] (0.5,2.5) circle (.065);
\draw[black,fill=white,line width=.45pt] (7.5,0.5) circle (.065);
\draw[black,line width=.65pt,rounded corners=1.5pt] (0.5,1.5)--(2.5,1.5);
\draw[black,fill=white,line width=.45pt] (0.5,1.5) circle (.065);
\draw[black,fill=white,line width=.45pt] (2.5,1.5) circle (.065);
\draw[black,line width=.65pt,rounded corners=1.5pt] (0.5,0.5)--(6.5,0.5);
\draw[black,fill=white,line width=.45pt] (0.5,0.5) circle (.065);
\draw[black,fill=white,line width=.45pt] (6.5,0.5) circle (.065);
\node[anchor=east] at (-.6,1.5) {$T_6$};
\node[anchor=north,font=\footnotesize] at (4,-.4) {$\mu_6=(k+2,k-1,3)$};

\end{scope}
\end{tikzpicture}
\caption{The six tails of shape $(k,k,4)$, drawn schematically with contents $\mu_1,\ldots,\mu_6$.
Both lower rows have length $k$; paths indicate the rim hooks.
The signs are positive for $T_1,T_5,T_6$ and negative for $T_2,T_3,T_4$.}
\label{fig:six-ribbons}
\label{tab:six}
\end{figure}

We next enumerate the semiordered chain partitions of these six types.

\begin{proposition}\label{prop:counts}
Let $n\ge4$ and $m\ge3n-1$, and put $k=m-n$. Suppose that $\lambda\vdash mn$ satisfies $\lambda_i=m+n-2i+1$ for $1\le i\le n-2$. For $j\in[6]$, if $\lambda_{\langle n-1\rangle}=\mu_j$, where $\mu_j$ is listed in \cref{tab:six}, then
\begin{equation}\label{eq:counts-prop}
\bigl|\operatorname{SCP}_{\bchain m\times\bchain n,\lambda}\bigr|
=(n-2)!F_j(n,k),
\end{equation}
where
\begin{equation}\label{eq:gap-count}
F_j(n,k)=\sum_{d=1}^{n-1}(n-d)B_j(n,k,d)
 +\sum_{r=1}^{3}\varepsilon_{j,r}(n,k).
\end{equation}
The polynomials $B_j=B_j(n,k,d)$ are
\begin{align}
6B_1={}&d^4+(2-4k)d^3+(6k^2-6k+5)d^2\notag\\*
&+(-4k^3+6k^2+38k-188)d+2k^4+12k^3+28k^2+30k+192n-180,\label{eq:B1}\\
12B_2={}&d^4+(2-4k)d^3+(6k^2-6k-1)d^2\notag\\*
&+(-4k^3+6k^2+74k-302)d+4k^4+20k^3+32k^2+16k+324n-324,\label{eq:B2}\\
B_3={}&(2k-3)d+3k^2+9k+11n-5,\label{eq:B3}\\
3B_4={}&(18k-48)d+4k^3+21k^2+35k+72n-54,\label{eq:B4}\\
B_5={}&2\bigl((2k-3)d+2k^2+8k+9n-2\bigr),\label{eq:B5}\\
B_6={}&(2k-6)d+k^3+3k^2+2k+12n-12,\label{eq:B6}
\end{align}
and the corrections $\varepsilon_{j,r}$ are given in \cref{tab:corrections}.
Here $j$ always labels a tail in \cref{tab:six}; $d$ denotes a row gap, and $r\in[3]$ labels its exceptional values. For example, $\mu_1=(k,k,4)=\langle4^1,k^2\rangle$. The proposition concerns exactly the six listed tails.
\end{proposition}

\begin{table}[htbp]\centering\small
\caption{Small-gap corrections; all corrections for $d\ge4$ vanish.}\label{tab:corrections}
\setlength{\tabcolsep}{3.3pt}
\renewcommand{\arraystretch}{1.2}
\begin{tabular}{@{}c*{6}{c}@{}}\toprule
$j$&$1$&$2$&$3$&$4$&$5$&$6$\\\midrule
$\varepsilon_{j,1}$&$\begin{gathered}2(n-1)(k^2+7k\\+14n-22)\end{gathered}$&$\begin{gathered}(n-1)(k^2+5k\\+17n-20)\end{gathered}$&$n(n-1)$&$\begin{gathered}2(n-1)\\(2k+4n-3)\end{gathered}$&$2n(n-1)$&$2n(n-1)$\\[2pt]
$\varepsilon_{j,2}$&$\begin{gathered}2(n-2)\\(2k+5n-9)\end{gathered}$&$\begin{gathered}(n-2)\\(2k+5n-11)\end{gathered}$&$0$&$2(n-2)(n-1)$&$0$&$0$\\[2pt]
$\varepsilon_{j,3}$&$2(n-3)(n-2)$&$(n-3)(n-2)$&$0$&$0$&$0$&$0$\\\bottomrule
\end{tabular}
\end{table}

\begin{proof}

We first parametrize the long chains and characterize their possible complements. We then count the three-chain partitions of the complements by separating a middle interval from two boundary regions. Finally, we evaluate the formulas for the six tail types.

\medskip
\noindent\textit{Step 1. Parametrizing the long chains.}
Fix $j\in[6]$, put $\mu=\mu_j$, and let $\lambda$ be the corresponding
full type in the proposition. In a semiordered chain partition of type
$\lambda$, let $C_i$ be the chain of size $\lambda_i$ for
$1\le i\le n-2$. We first parametrize these long chains and then
partition their complement. By the equality argument after
\eqref{eq:capacity}, each $C_i$ is saturated and occupies exactly
ranks $i-1,\ldots,m+n-i-1$.
On their common ranks $n-3,\ldots,m+1$, the paths preserve left-to-right
order: a change in their order would force two saturated paths to share a vertex.
Thus the $i$th path $Q_i$ joins $(i,n-1-i)$ to
$(m-n+2+i,n+1-i)$. It has exactly two vertical steps, at first
coordinates $r_i\le s_i$, and hence
\begin{equation}\label{eq:central-paths}
\begin{aligned}
Q_i={}&\{(x,n-1-i):i\le x\le r_i\}\cup\{(x,n-i):r_i\le x\le s_i\}\\
&\cup\{(x,n+1-i):s_i\le x\le m-n+2+i\}.
\end{aligned}
\end{equation}
Comparing adjacent paths on their two common rows gives
$r_i<r_{i+1}$ and $s_i<s_{i+1}$. The only nonadjacent paths sharing
a row are $Q_i,Q_{i+2}$; their intervals are disjoint because
$r_i<s_{i+2}$. Consequently the path data are exactly
\begin{equation}\label{eq:upsilon}
\Upsilon=(\mathbf r,\mathbf s),\qquad
1\le r_1<\cdots<r_{n-2}\le m,\quad
1\le s_1<\cdots<s_{n-2}\le m,\quad r_i\le s_i.
\end{equation}
Conversely, any $\Upsilon$ satisfying \eqref{eq:upsilon} has
$i\le r_i\le s_i\le m-n+2+i$, so \eqref{eq:central-paths}
gives disjoint paths with the required endpoints.
Equality $r_i=s_i$ is allowed.

For any $\sigma\in\mathfrak S_{n-2}$, assign $Q_i$ to
$C_{\sigma_i}$. At each of ranks $q-1$ and
$m+n-q-1$, $1\le q\le n-2$, place the active labels $1,\ldots,q$
from left to right in their order in $\sigma$. When $q$ changes by
one, each existing label shifts by at most one position. Each shift
gives a saturated step. We can therefore extend the paths at both ends
to disjoint chains of the prescribed lengths. The extension is unique,
since disjoint paths preserve order and the active labels are fixed. Thus the long-chain families are in bijection with the pairs
$(\sigma,\Upsilon)$.
Their complement is the induced poset
\[
R_\Upsilon=
\left\{(x,y)\in[m]\times[n]:n-2\le x+y-2\le m\right\}
\mathbin{\setminus}\bigcup_{i=1}^{n-2}Q_i,
\]
independent of $\sigma$. Since every part of $\mu$ is at most
$k+2<k+5=\lambda_{n-2}$, removing the long chains and adjoining them
again are inverse operations on semiordered partitions. Both operations
preserve the order of equal-sized tail chains. Therefore
\begin{equation}\label{eq:factor-yang}
\bigl|\operatorname{SCP}_{\bchain m\times\bchain n,\lambda}\bigr|
=(n-2)!\sum_\Upsilon
\bigl|\operatorname{SCP}_{R_\Upsilon,\mu}\bigr|.
\end{equation}
It remains to evaluate
\begin{equation}\label{eq:Fi-def}
F_j(n,k)=\sum_\Upsilon\bigl|\operatorname{SCP}_{R_\Upsilon,\mu_j}\bigr|.
\end{equation}

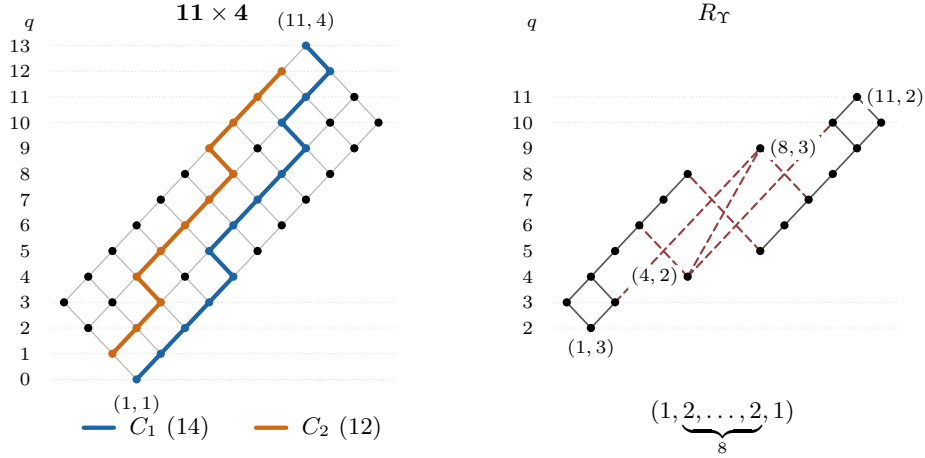
\begin{figure}[htbp]\centering
\begingroup
\definecolor{yrBlue}{RGB}{29,100,165}
\definecolor{yrOrange}{RGB}{205,105,20}
\definecolor{yrPurple}{RGB}{122,76,154}
\definecolor{yrGreen}{RGB}{64,135,107}
\definecolor{yrGold}{RGB}{209,163,48}
\definecolor{yrRed}{RGB}{155,60,63}

\begin{tikzpicture}[x=1cm,y=1cm,line cap=round,line join=round,font=\small]
\begin{scope}[xshift=0.0cm]
\node at (1.00,4.85) {$\mathbf{11}\times\mathbf4$};
\node[font=\scriptsize] at (-1.42,4.70) {$q$};
\node[font=\scriptsize,anchor=east] at (-1.28,0.000) {$ 0 $};
\draw[gray!22,densely dotted,line width=.3pt] (-1.12,0.000) -- (3.43,0.000);
\node[font=\scriptsize,anchor=east] at (-1.28,0.340) {$ 1 $};
\draw[gray!22,densely dotted,line width=.3pt] (-1.12,0.340) -- (3.43,0.340);
\node[font=\scriptsize,anchor=east] at (-1.28,0.680) {$ 2 $};
\draw[gray!22,densely dotted,line width=.3pt] (-1.12,0.680) -- (3.43,0.680);
\node[font=\scriptsize,anchor=east] at (-1.28,1.020) {$ 3 $};
\draw[gray!22,densely dotted,line width=.3pt] (-1.12,1.020) -- (3.43,1.020);
\node[font=\scriptsize,anchor=east] at (-1.28,1.360) {$ 4 $};
\node[font=\scriptsize,anchor=east] at (-1.28,1.700) {$ 5 $};
\node[font=\scriptsize,anchor=east] at (-1.28,2.040) {$ 6 $};
\node[font=\scriptsize,anchor=east] at (-1.28,2.380) {$ 7 $};
\node[font=\scriptsize,anchor=east] at (-1.28,2.720) {$ 8 $};
\node[font=\scriptsize,anchor=east] at (-1.28,3.060) {$ 9 $};
\node[font=\scriptsize,anchor=east] at (-1.28,3.400) {$ 10 $};
\draw[gray!22,densely dotted,line width=.3pt] (-1.12,3.400) -- (3.43,3.400);
\node[font=\scriptsize,anchor=east] at (-1.28,3.740) {$ 11 $};
\draw[gray!22,densely dotted,line width=.3pt] (-1.12,3.740) -- (3.43,3.740);
\node[font=\scriptsize,anchor=east] at (-1.28,4.080) {$ 12 $};
\draw[gray!22,densely dotted,line width=.3pt] (-1.12,4.080) -- (3.43,4.080);
\node[font=\scriptsize,anchor=east] at (-1.28,4.420) {$ 13 $};
\draw[gray!22,densely dotted,line width=.3pt] (-1.12,4.420) -- (3.43,4.420);
\draw[gray!58,line width=.45pt] (0.000,0.000) -- (-0.320,0.340);
\draw[gray!58,line width=.45pt] (0.000,0.000) -- (0.320,0.340);
\draw[gray!58,line width=.45pt] (-0.320,0.340) -- (-0.640,0.680);
\draw[gray!58,line width=.45pt] (-0.320,0.340) -- (0.000,0.680);
\draw[gray!58,line width=.45pt] (-0.640,0.680) -- (-0.960,1.020);
\draw[gray!58,line width=.45pt] (-0.640,0.680) -- (-0.320,1.020);
\draw[gray!58,line width=.45pt] (-0.960,1.020) -- (-0.640,1.360);
\draw[gray!58,line width=.45pt] (0.320,0.340) -- (0.000,0.680);
\draw[gray!58,line width=.45pt] (0.320,0.340) -- (0.640,0.680);
\draw[gray!58,line width=.45pt] (0.000,0.680) -- (-0.320,1.020);
\draw[gray!58,line width=.45pt] (0.000,0.680) -- (0.320,1.020);
\draw[gray!58,line width=.45pt] (-0.320,1.020) -- (-0.640,1.360);
\draw[gray!58,line width=.45pt] (-0.320,1.020) -- (0.000,1.360);
\draw[gray!58,line width=.45pt] (-0.640,1.360) -- (-0.320,1.700);
\draw[gray!58,line width=.45pt] (0.640,0.680) -- (0.320,1.020);
\draw[gray!58,line width=.45pt] (0.640,0.680) -- (0.960,1.020);
\draw[gray!58,line width=.45pt] (0.320,1.020) -- (0.000,1.360);
\draw[gray!58,line width=.45pt] (0.320,1.020) -- (0.640,1.360);
\draw[gray!58,line width=.45pt] (0.000,1.360) -- (-0.320,1.700);
\draw[gray!58,line width=.45pt] (0.000,1.360) -- (0.320,1.700);
\draw[gray!58,line width=.45pt] (-0.320,1.700) -- (0.000,2.040);
\draw[gray!58,line width=.45pt] (0.960,1.020) -- (0.640,1.360);
\draw[gray!58,line width=.45pt] (0.960,1.020) -- (1.280,1.360);
\draw[gray!58,line width=.45pt] (0.640,1.360) -- (0.320,1.700);
\draw[gray!58,line width=.45pt] (0.640,1.360) -- (0.960,1.700);
\draw[gray!58,line width=.45pt] (0.320,1.700) -- (0.000,2.040);
\draw[gray!58,line width=.45pt] (0.320,1.700) -- (0.640,2.040);
\draw[gray!58,line width=.45pt] (0.000,2.040) -- (0.320,2.380);
\draw[gray!58,line width=.45pt] (1.280,1.360) -- (0.960,1.700);
\draw[gray!58,line width=.45pt] (1.280,1.360) -- (1.600,1.700);
\draw[gray!58,line width=.45pt] (0.960,1.700) -- (0.640,2.040);
\draw[gray!58,line width=.45pt] (0.960,1.700) -- (1.280,2.040);
\draw[gray!58,line width=.45pt] (0.640,2.040) -- (0.320,2.380);
\draw[gray!58,line width=.45pt] (0.640,2.040) -- (0.960,2.380);
\draw[gray!58,line width=.45pt] (0.320,2.380) -- (0.640,2.720);
\draw[gray!58,line width=.45pt] (1.600,1.700) -- (1.280,2.040);
\draw[gray!58,line width=.45pt] (1.600,1.700) -- (1.920,2.040);
\draw[gray!58,line width=.45pt] (1.280,2.040) -- (0.960,2.380);
\draw[gray!58,line width=.45pt] (1.280,2.040) -- (1.600,2.380);
\draw[gray!58,line width=.45pt] (0.960,2.380) -- (0.640,2.720);
\draw[gray!58,line width=.45pt] (0.960,2.380) -- (1.280,2.720);
\draw[gray!58,line width=.45pt] (0.640,2.720) -- (0.960,3.060);
\draw[gray!58,line width=.45pt] (1.920,2.040) -- (1.600,2.380);
\draw[gray!58,line width=.45pt] (1.920,2.040) -- (2.240,2.380);
\draw[gray!58,line width=.45pt] (1.600,2.380) -- (1.280,2.720);
\draw[gray!58,line width=.45pt] (1.600,2.380) -- (1.920,2.720);
\draw[gray!58,line width=.45pt] (1.280,2.720) -- (0.960,3.060);
\draw[gray!58,line width=.45pt] (1.280,2.720) -- (1.600,3.060);
\draw[gray!58,line width=.45pt] (0.960,3.060) -- (1.280,3.400);
\draw[gray!58,line width=.45pt] (2.240,2.380) -- (1.920,2.720);
\draw[gray!58,line width=.45pt] (2.240,2.380) -- (2.560,2.720);
\draw[gray!58,line width=.45pt] (1.920,2.720) -- (1.600,3.060);
\draw[gray!58,line width=.45pt] (1.920,2.720) -- (2.240,3.060);
\draw[gray!58,line width=.45pt] (1.600,3.060) -- (1.280,3.400);
\draw[gray!58,line width=.45pt] (1.600,3.060) -- (1.920,3.400);
\draw[gray!58,line width=.45pt] (1.280,3.400) -- (1.600,3.740);
\draw[gray!58,line width=.45pt] (2.560,2.720) -- (2.240,3.060);
\draw[gray!58,line width=.45pt] (2.560,2.720) -- (2.880,3.060);
\draw[gray!58,line width=.45pt] (2.240,3.060) -- (1.920,3.400);
\draw[gray!58,line width=.45pt] (2.240,3.060) -- (2.560,3.400);
\draw[gray!58,line width=.45pt] (1.920,3.400) -- (1.600,3.740);
\draw[gray!58,line width=.45pt] (1.920,3.400) -- (2.240,3.740);
\draw[gray!58,line width=.45pt] (1.600,3.740) -- (1.920,4.080);
\draw[gray!58,line width=.45pt] (2.880,3.060) -- (2.560,3.400);
\draw[gray!58,line width=.45pt] (2.880,3.060) -- (3.200,3.400);
\draw[gray!58,line width=.45pt] (2.560,3.400) -- (2.240,3.740);
\draw[gray!58,line width=.45pt] (2.560,3.400) -- (2.880,3.740);
\draw[gray!58,line width=.45pt] (2.240,3.740) -- (1.920,4.080);
\draw[gray!58,line width=.45pt] (2.240,3.740) -- (2.560,4.080);
\draw[gray!58,line width=.45pt] (1.920,4.080) -- (2.240,4.420);
\draw[gray!58,line width=.45pt] (3.200,3.400) -- (2.880,3.740);
\draw[gray!58,line width=.45pt] (2.880,3.740) -- (2.560,4.080);
\draw[gray!58,line width=.45pt] (2.560,4.080) -- (2.240,4.420);
\draw[yrBlue,line width=1.6pt] (0.000,0.000) -- (0.320,0.340) -- (0.640,0.680) -- (0.960,1.020) -- (1.280,1.360) -- (0.960,1.700) -- (1.280,2.040) -- (1.600,2.380) -- (1.920,2.720) -- (2.240,3.060) -- (1.920,3.400) -- (2.240,3.740) -- (2.560,4.080) -- (2.240,4.420);
\draw[yrOrange,line width=1.6pt] (-0.320,0.340) -- (0.000,0.680) -- (0.320,1.020) -- (0.000,1.360) -- (0.320,1.700) -- (0.640,2.040) -- (0.960,2.380) -- (1.280,2.720) -- (0.960,3.060) -- (1.280,3.400) -- (1.600,3.740) -- (1.920,4.080);
\fill[yrBlue] (0.000,0.000) circle (1.55pt);
\fill[yrOrange] (-0.320,0.340) circle (1.55pt);
\fill[black] (-0.640,0.680) circle (1.55pt);
\fill[black] (-0.960,1.020) circle (1.55pt);
\fill[yrBlue] (0.320,0.340) circle (1.55pt);
\fill[yrOrange] (0.000,0.680) circle (1.55pt);
\fill[black] (-0.320,1.020) circle (1.55pt);
\fill[black] (-0.640,1.360) circle (1.55pt);
\fill[yrBlue] (0.640,0.680) circle (1.55pt);
\fill[yrOrange] (0.320,1.020) circle (1.55pt);
\fill[yrOrange] (0.000,1.360) circle (1.55pt);
\fill[black] (-0.320,1.700) circle (1.55pt);
\fill[yrBlue] (0.960,1.020) circle (1.55pt);
\fill[black] (0.640,1.360) circle (1.55pt);
\fill[yrOrange] (0.320,1.700) circle (1.55pt);
\fill[black] (0.000,2.040) circle (1.55pt);
\fill[yrBlue] (1.280,1.360) circle (1.55pt);
\fill[yrBlue] (0.960,1.700) circle (1.55pt);
\fill[yrOrange] (0.640,2.040) circle (1.55pt);
\fill[black] (0.320,2.380) circle (1.55pt);
\fill[black] (1.600,1.700) circle (1.55pt);
\fill[yrBlue] (1.280,2.040) circle (1.55pt);
\fill[yrOrange] (0.960,2.380) circle (1.55pt);
\fill[black] (0.640,2.720) circle (1.55pt);
\fill[black] (1.920,2.040) circle (1.55pt);
\fill[yrBlue] (1.600,2.380) circle (1.55pt);
\fill[yrOrange] (1.280,2.720) circle (1.55pt);
\fill[yrOrange] (0.960,3.060) circle (1.55pt);
\fill[black] (2.240,2.380) circle (1.55pt);
\fill[yrBlue] (1.920,2.720) circle (1.55pt);
\fill[black] (1.600,3.060) circle (1.55pt);
\fill[yrOrange] (1.280,3.400) circle (1.55pt);
\fill[black] (2.560,2.720) circle (1.55pt);
\fill[yrBlue] (2.240,3.060) circle (1.55pt);
\fill[yrBlue] (1.920,3.400) circle (1.55pt);
\fill[yrOrange] (1.600,3.740) circle (1.55pt);
\fill[black] (2.880,3.060) circle (1.55pt);
\fill[black] (2.560,3.400) circle (1.55pt);
\fill[yrBlue] (2.240,3.740) circle (1.55pt);
\fill[yrOrange] (1.920,4.080) circle (1.55pt);
\fill[black] (3.200,3.400) circle (1.55pt);
\fill[black] (2.880,3.740) circle (1.55pt);
\fill[yrBlue] (2.560,4.080) circle (1.55pt);
\fill[yrBlue] (2.240,4.420) circle (1.55pt);
\node[font=\scriptsize,anchor=north] at (0.000,-0.080) {$(1,1)$};
\node[font=\scriptsize,anchor=south] at (2.240,4.500) {$(11,4)$};
\draw[yrBlue,line width=1.6pt] (-.70,-.64)--(-.30,-.64);
\node[anchor=west] at (-.22,-.64) {$C_1\ (14)$};
\draw[yrOrange,line width=1.6pt] (1.58,-.64)--(1.98,-.64);
\node[anchor=west] at (2.06,-.64) {$C_2\ (12)$};
\end{scope}
\begin{scope}[xshift=6.65cm]
\node at (1.00,4.85) {$R_\Upsilon$};
\node[font=\scriptsize] at (-1.42,4.70) {$q$};
\node[font=\scriptsize,anchor=east] at (-1.28,0.680) {$ 2 $};
\draw[gray!22,densely dotted,line width=.3pt] (-1.12,0.680) -- (3.43,0.680);
\node[font=\scriptsize,anchor=east] at (-1.28,1.020) {$ 3 $};
\draw[gray!22,densely dotted,line width=.3pt] (-1.12,1.020) -- (3.43,1.020);
\node[font=\scriptsize,anchor=east] at (-1.28,1.360) {$ 4 $};
\node[font=\scriptsize,anchor=east] at (-1.28,1.700) {$ 5 $};
\node[font=\scriptsize,anchor=east] at (-1.28,2.040) {$ 6 $};
\node[font=\scriptsize,anchor=east] at (-1.28,2.380) {$ 7 $};
\node[font=\scriptsize,anchor=east] at (-1.28,2.720) {$ 8 $};
\node[font=\scriptsize,anchor=east] at (-1.28,3.060) {$ 9 $};
\node[font=\scriptsize,anchor=east] at (-1.28,3.400) {$ 10 $};
\draw[gray!22,densely dotted,line width=.3pt] (-1.12,3.400) -- (3.43,3.400);
\node[font=\scriptsize,anchor=east] at (-1.28,3.740) {$ 11 $};
\draw[gray!22,densely dotted,line width=.3pt] (-1.12,3.740) -- (3.43,3.740);
\draw[black!65,line width=.65pt] (-0.640,0.680) -- (-0.960,1.020);
\draw[black!65,line width=.65pt] (-0.640,0.680) -- (-0.320,1.020);
\draw[black!65,line width=.65pt] (-0.960,1.020) -- (-0.640,1.360);
\draw[black!65,line width=.65pt] (-0.320,1.020) -- (-0.640,1.360);
\draw[yrRed,densely dashed,line width=.8pt] (-0.320,1.020) -- (1.600,3.060);
\draw[black!65,line width=.65pt] (-0.640,1.360) -- (-0.320,1.700);
\draw[black!65,line width=.65pt] (-0.320,1.700) -- (0.000,2.040);
\draw[yrRed,densely dashed,line width=.8pt] (0.640,1.360) -- (0.000,2.040);
\draw[yrRed,densely dashed,line width=.8pt] (0.640,1.360) -- (1.600,3.060);
\draw[yrRed,densely dashed,line width=.8pt] (0.640,1.360) -- (2.560,3.400);
\draw[black!65,line width=.65pt] (0.000,2.040) -- (0.320,2.380);
\draw[black!65,line width=.65pt] (0.320,2.380) -- (0.640,2.720);
\draw[yrRed,densely dashed,line width=.8pt] (1.600,1.700) -- (0.640,2.720);
\draw[black!65,line width=.65pt] (1.600,1.700) -- (1.920,2.040);
\draw[black!65,line width=.65pt] (1.920,2.040) -- (2.240,2.380);
\draw[yrRed,densely dashed,line width=.8pt] (2.240,2.380) -- (1.600,3.060);
\draw[black!65,line width=.65pt] (2.240,2.380) -- (2.560,2.720);
\draw[black!65,line width=.65pt] (2.560,2.720) -- (2.880,3.060);
\draw[black!65,line width=.65pt] (2.880,3.060) -- (2.560,3.400);
\draw[black!65,line width=.65pt] (2.880,3.060) -- (3.200,3.400);
\draw[black!65,line width=.65pt] (2.560,3.400) -- (2.880,3.740);
\draw[black!65,line width=.65pt] (3.200,3.400) -- (2.880,3.740);
\fill[black] (-0.640,0.680) circle (1.55pt);
\fill[black] (-0.960,1.020) circle (1.55pt);
\fill[black] (-0.320,1.020) circle (1.55pt);
\fill[black] (-0.640,1.360) circle (1.55pt);
\fill[black] (-0.320,1.700) circle (1.55pt);
\fill[black] (0.640,1.360) circle (1.55pt);
\fill[black] (0.000,2.040) circle (1.55pt);
\fill[black] (0.320,2.380) circle (1.55pt);
\fill[black] (1.600,1.700) circle (1.55pt);
\fill[black] (0.640,2.720) circle (1.55pt);
\fill[black] (1.920,2.040) circle (1.55pt);
\fill[black] (2.240,2.380) circle (1.55pt);
\fill[black] (1.600,3.060) circle (1.55pt);
\fill[black] (2.560,2.720) circle (1.55pt);
\fill[black] (2.880,3.060) circle (1.55pt);
\fill[black] (2.560,3.400) circle (1.55pt);
\fill[black] (3.200,3.400) circle (1.55pt);
\fill[black] (2.880,3.740) circle (1.55pt);
\node[font=\scriptsize,anchor=north,fill=white,inner sep=1pt] at (-0.640,0.570) {$(1,3)$};
\node[font=\scriptsize,anchor=west,fill=white,inner sep=1pt] at (2.960,3.740) {$(11,2)$};
\node[font=\scriptsize,anchor=east,fill=white,inner sep=1pt] at (0.560,1.360) {$(4,2)$};
\node[font=\scriptsize,anchor=west,fill=white,inner sep=1pt] at (1.680,3.060) {$(8,3)$};
\node at (1.10,-.64) {$(1,\underbrace{2,\ldots,2}_{8},1)$};
\end{scope}
\end{tikzpicture}
\endgroup

\caption{Long chains and their complement in $\bchain{11}\times\bchain4$, with $\sigma=(2,1)$ and $\Upsilon=((3,5),(7,9))$. The labels $q$ are original ranks; dashed red cover edges of $R_\Upsilon$ join nonconsecutive original ranks.}\label{fig:poset-chains}
\end{figure}

\smallskip\noindent\emph{Step 2. Characterizing the complement.}
The complement has one vertex at each of ranks $n-2,m$ and two at
every intervening rank. Number these levels $t=0,\ldots,k+2$;
a vertex in row $y$ at level $t$ is $(n+t-y,y)$.
Write $h_0\in[n-1]$ and $h_{k+2}\in\{2,\ldots,n\}$ for the singleton
rows, and $1\le p_t<q_t\le n$ for the two rows at each internal level. For $1\le t\le k+1$, the vertices in rows $p_t$ and $q_t$
are called the lower and upper vertices of level $t$, respectively.
We show that sending $\Upsilon$ to the rows of its complement gives
a bijection onto the row data satisfying
\begin{equation}\label{eq:interlace}
p_1\le h_0<q_1,\quad
p_{t+1}\le p_t<q_{t+1}\le q_t\quad(1\le t\le k),\quad
p_{k+1}<h_{k+2}\le q_{k+1}.
\end{equation}

At two consecutive internal levels, let $(p,q)$ and $(p',q')$ be
the rows not used by the long chains. A long-chain vertex in row $u$
must be matched to row $u$ or $u+1$. If $p'>p$, the next level requires $p$ occupied vertices
in rows at most $p$, but only $p-1$ possible predecessors exist.
Similarly, once $p'\le p$, the inequality $q'>q$ would require
$q-1$ predecessors in rows at most $q$, where only $q-2$ exist.
Finally, if $q'\le p$, the $p-1$ occupied source rows below $p$
would have to fit into only $p-2$ occupied target rows at most $p$.
Thus $p'\le p<q'\le q$. Conversely, under these inequalities the map
\[
f(u)=
\begin{cases}
u+1,&p'\le u<p\ \text{or}\ q'\le u<q,\\
u,&\text{otherwise},
\end{cases}
\qquad u\in[n]\setminus\{p,q\},
\]
is a bijection onto $[n]\setminus\{p',q'\}$: its inverse subtracts
one on $(p',p]\cup(q',q]$ and fixes all other rows.
Any legal matching $g$ satisfies $g(u)\le u+1\le v\le g(v)$
for $u<v$; injectivity makes it strictly increasing, so it is
the unique increasing bijection of the occupied row sets.

Treat the nonexistent rows as vacancies. Apply
$p'\le p<q'\le q$ to $(h_0,n)$ and $(p_1,q_1)$ at the lower end,
and to $(p_{k+1},q_{k+1})$ and $(1,h_{k+2})$ at the upper end.
This gives the first and last inequalities in \eqref{eq:interlace}.
The matching at each end is the unique increasing bijection of the occupied row sets.
Extend to ranks $n-3$ and $m+1$
by shifting row $u$ up one exactly when $u\ge h_0$ at the lower
end and $u<h_{k+2}$ at the upper end. Order preservation pairs
the occupied rows $1,\ldots,n-2$ with $3,\ldots,n$, respectively.
Thus each path raises its row by two, has exactly two vertical
steps, and has the endpoints of Step~1. Their vertical-step
positions recover an admissible $\Upsilon$. Each matching is unique,
so the row data determine all the paths. Conversely, the constructed
paths occupy exactly the nonvacant rows. Hence the map from paths to
row data and the construction of paths from row data are inverse.

Label the three tail chains $A,B,C$ in the prescribed order,
including the order among equal sizes. Vertices at levels $t<t'$
and rows $y,y'$ are comparable precisely when
\begin{equation}\label{eq:color-chain}
0\le y'-y\le t'-t.
\end{equation}
Thus a valid tail partition is exactly a labelling with distinct
labels at each internal level, the prescribed label counts, and
\eqref{eq:color-chain} between consecutive occurrences of each label.
All other comparisons follow by transitivity.

\smallskip\noindent\emph{Step 3. Fixing the middle rows and the boundaries.}
Fix a valid tail partition. An internal level is an $AB$-level if its
labels are $A,B$; every other internal level contains $C$. Since
$|C|\le4$ and $k\ge7$, there are at least $k-3\ge4$ $AB$-levels.
The order of $A,B$ is constant at these levels. Indeed, if $A$ is
below $B$ at level $t$ but above it at a later $AB$-level $t'$, then
the row of $B$ decreases: \eqref{eq:interlace} gives
$p_{t'}\le p_t<q_t$, contradicting \eqref{eq:color-chain}. The reverse initial order is symmetric.
We first count the order with $A$ below $B$; the other will be
included by exchanging their labels.

Let their rows at the first $AB$-level be $a<b$, and set $d=b-a$.
At subsequent $AB$-levels, each row is weakly decreasing by
\eqref{eq:interlace} and weakly increasing along its fixed chain by
\eqref{eq:color-chain}; hence the rows remain $a,b$. At every
intervening level, \eqref{eq:interlace} gives
$a\ge p_t\ge a$ and $b\ge q_t\ge b$. Thus the two rows of the
complement are fixed between the first and last $AB$-levels. Moreover, every intervening vertex
of $A$ has row $a$, and every such vertex of $B$ has row $b$, by
applying \eqref{eq:color-chain} between that vertex and the two endpoint
$AB$-levels.
Consequently the middle labels are $AB$, $BC$ with $C$ in row $a$,
or $AC$ with $C$ in row $b$.

Let $p,s$ count the internal levels before the first and after the
last $AB$-level. The lower boundary consists of the $p$ internal levels
and singleton level $0$. The upper boundary consists of the $s$ internal
levels and singleton level $k+2$. Even when a boundary has no internal
levels, it still contains its singleton. Each internal boundary level contains $C$,
so
\begin{equation}\label{eq:ps}
p,s\ge0,\qquad p+s\le4.
\end{equation}
The middle consists of the levels strictly between the first and
last $AB$-levels, excluding these endpoints, and has length
\begin{equation}\label{eq:M}
M=k-1-p-s.
\end{equation}
All middle occurrences of $C$ in row $a$ precede those in row $b$;
the transition between these rows must also span at least $d$ levels
by \eqref{eq:color-chain}. We use these restrictions to count the
middle configurations in Step~5.

Assign weights $z,z^{-1},w$ to vertices labelled $A,B,C$, respectively.
A labelling then has weight
\begin{equation}\label{eq:weight}
z^{\#A-\#B}w^{\#C}.
\end{equation}
The fixed total $\#A+\#B+\#C=2k+4$ means that these two exponents
determine all three chain sizes. An $AB$, $AC$, or $BC$ level has
weight $1$, $zw$, or $z^{-1}w$, respectively. We work in
$\mathbb Q[z,z^{-1}][w]$ and write
$\overline f(z,w)=f(z^{-1},w)$ for the exchange of $A,B$.
We only need coefficients of $w$-degree at most four. We may therefore
compute modulo $(w^5)$; terms of higher $w$-degree do not affect the
required coefficients.

\smallskip\noindent\emph{Step 4. Counting a single boundary.}
We count the lower boundary, including its singleton, according to the
row of its last vertex labelled $C$. Reflection will give the upper
boundary in Step~5.
Subtract $a$ from each row coordinate and call the result the relative
row. Thus the first $AB$-level, at level $p+1$, has relative rows
$0,d$. The boundary consists of levels $0,\ldots,p$. Its weight
does not include the first $AB$-level, whose weight is $1$.
In Step~4, all rows refer to relative rows unless stated otherwise.
By \eqref{eq:interlace}, all boundary rows are nonnegative, and
the upper row at each internal boundary level is at least $d$.
If the boundary contains a vertex labelled $C$, let $\xi\ge0$
be the row of its last such vertex. Otherwise, set $\xi=\bot$. Let $b_p(\xi;d)$ be the sum of the
weights of boundaries with this value of $\xi$. For $p\ge1$, each
internal boundary level contains one $C$, so the last $C$ is at level
$p$. For $p=0$, it is the singleton if labelled $C$.

When $p=0$, the boundary consists only of its singleton.
By the first inequality in \eqref{eq:interlace}, its row belongs to
$\{0,\ldots,d-1\}$. Applying \eqref{eq:color-chain} between levels
$0$ and $1$ forces a singleton $A$ to have row $0$ and a singleton
$B$ to have row $d-1$.
A singleton $C$ has one position for each fixed allowed $\xi$.
For each $p$, let $N_p$ be the weight sum when there is no $C$.
Let $R_p$ be the contribution common to all $0\le\xi<d$, with an
additional contribution $D_p$ at $\xi=0$. Let $U_p,V_p$ be the weight
sums at $\xi=d$ and at each fixed $\xi>d$, respectively. For $p=0$,
we obtain
\begin{equation}\label{eq:profile-zero}
N_0=z+z^{-1},\qquad R_0=w,\qquad D_0=U_0=V_0=0.
\end{equation}

Now let $1\le p\le4$. Consecutive $C$'s in the same position
(lower or upper within their level) have equal rows. Indeed,
\eqref{eq:interlace} makes their rows weakly decreasing, whereas
\eqref{eq:color-chain} makes them weakly increasing.
An upper $C$ cannot be followed by a lower one, since
\eqref{eq:interlace} would make its row decrease. If $C$ changes
from the lower to the upper position, \eqref{eq:interlace} makes
its row strictly increase. By \eqref{eq:color-chain}, the increase
between consecutive levels is at most one, so it is exactly one.
Hence there are precisely three cases: $C$ is always lower, always
upper, or lower then upper.
By \eqref{eq:color-chain}, every internal $A$ has row at most
$0$, the row of $A$ at the first $AB$-level. Since all boundary rows
are nonnegative, every internal $A$ has row $0$ and is lower.
Every upper $B$ has row $d$: \eqref{eq:color-chain} gives the upper
bound $d$, and \eqref{eq:interlace} gives the lower bound $d$.

\smallskip\noindent\emph{Always lower.}
The internal pairs are $(\xi,d)$, with labels $C,B$, and
$0\le\xi<d$. The singleton row lies in $[\xi,d-1]$.
A singleton $B$ must have row $d-1$, and a singleton $C$ row $\xi$;
a singleton $A$ in row $0$ is allowed exactly when $\xi=0$.
These choices satisfy \eqref{eq:interlace} and
\eqref{eq:color-chain}, giving
\begin{equation}\label{eq:profile-low}
N_p=0,\qquad
R_p=w^pz^{-p-1}+w^{p+1}z^{-p},\qquad D_p=w^pz^{1-p}
\quad(1\le p\le4).
\end{equation}

\smallskip\noindent\emph{Always upper.}
Here $C$ has constant row $\xi\ge d$. The lower labels form a word
in $A,B$. All lower $B$'s, when present, have the same row $v$:
\eqref{eq:interlace} makes their rows weakly decreasing, and
\eqref{eq:color-chain} makes them weakly increasing. Thus $0\le v\le d$ and
$v<\xi$.
If their number is $\nu$, the internal weight is
$w^pz^{p-2\nu}$. To record the singleton choices, set
\begin{equation}\label{eq:Phi}
H_p(d)=\max\{0,p+1-d\},\quad
\Phi_{p,\nu}=w^pz^{p-2\nu-1}+w^{p+1}z^{p-2\nu},\quad
\Psi_{p,\nu}=w^pz^{p-2\nu+1}+\Phi_{p,\nu}.
\end{equation}
Here $\Phi_{p,\nu}$ includes one singleton choice for each of
the labels $B,C$, and $\Psi_{p,\nu}$ includes one singleton
choice for each of the labels $A,B,C$. We determine the singleton positions in the four disjoint cases (a)--(d).
All row intervals in these cases contain only integers.

\emph{(a) No lower $B$.}
The word is $A^p$, with pairs $(0,\xi)$. The singleton $A$ and $C$
have rows $0$ and $\xi-1$, respectively. A singleton $B$ must
satisfy \eqref{eq:interlace} at level $0$ and
\eqref{eq:color-chain} with the $B$ in row $d$ at level $p+1$.
Its row therefore lies in
$[\max(0,d-p-1),\min(d,\xi-1)]$.
There are $\min(d,p+1)$ choices when $\xi=d$, and one more when
$\xi>d$. The contribution is therefore
\[
w^pz^{p+1}
+\bigl(\min(d,p+1)+\ind{\xi>d}\bigr)w^pz^{p-1}
+w^{p+1}z^p.
\]
Here $\ind{E}$ equals $1$ if the condition $E$ holds and $0$
otherwise.

\emph{(b) Lower $B$-row $v=0$.}
If the last lower $B$ is at level $t$, applying
\eqref{eq:color-chain} to this vertex and the $B$ in row $d$ at
level $p+1$ gives $d\le p+1-t$, or $t\le p+1-d$.
Conversely, choose any nonempty subset of levels $1,\ldots,H_p(d)$
for the lower $B$'s, and put $A$ at the remaining lower positions.
All lower rows are $0$. The bound on $t$ ensures that the last lower
$B$ is comparable with the $B$ at the first $AB$-level.
For $\nu$ such levels there are $\binom{H_p(d)}\nu$ choices.
The singleton has row $0$, $0$, or $\xi-1$ when labelled $A$, $B$,
or $C$, respectively. The contribution is
$\sum_{\nu=1}^{H_p(d)}\binom{H_p(d)}\nu\Psi_{p,\nu}$.

\emph{(c) Lower $B$-row $1\le v \le d-1$.}
By \eqref{eq:interlace}, the lower rows cannot increase, so the
word is $B^\nu A^{p-\nu}$, $1\le\nu\le p$.
Applying \eqref{eq:color-chain} between the last lower $B$ and the
$B$ at the first $AB$-level gives $d-v\le p+1-\nu$; hence
$v\in[\max(1,d-p-1+\nu),d-1]$, giving
$\min(d-1,p+1-\nu)$ choices (zero when $d=1$).
The pairs $(v,\xi)$ followed by $(0,\xi)$ satisfy
\eqref{eq:interlace}, including the transition to $(0,d)$ at the
first $AB$-level.
The bound on $v$ ensures that the vertices labelled $B$ form a chain.
The singleton may be $B$ in row $v$ or $C$ in row $\xi-1$,
but not $A$. The contribution is
$\sum_{\nu=1}^{p}\min(d-1,p+1-\nu)\Phi_{p,\nu}$.

\emph{(d) Lower $B$-row $v=d$.}
Now $\xi>d$, and the word is again $B^\nu A^{p-\nu}$.
If $\nu=p$, the transition from $(d,\xi)$ to $(0,d)$ would
require $d<d$ in \eqref{eq:interlace}. Every $1\le\nu<p$ is valid.
A singleton $B$ in row $d$ or $C$ in row $\xi-1$ gives the
contribution $\sum_{\nu=1}^{p-1}\Phi_{p,\nu}$, only for $\xi>d$.

\smallskip\noindent\emph{Lower then upper.}
Let the change occur after level $\nu$, where $1\le\nu<p$.
Before the change, $C$ has row below $d$; after the change, it
has row at least $d$. Since the row increases by one at the change,
the two rows are $d-1,d$, and $\xi=d$.
The first $\nu$ levels have labels $C,B$ in rows $d-1,d$.
After the change, the lower label must be $A$ in row $0$: a lower
$B$ would have row less than $d$ and follow an earlier $B$ in row
$d$, contrary to \eqref{eq:color-chain}.
There is thus one internal configuration for each $\nu$.
The singleton must have row $d-1$. It may be labelled $B$ or $C$,
and it may be labelled $A$ exactly when $d=1$.
These row pairs satisfy \eqref{eq:interlace}. The vertices labelled
$C$ satisfy \eqref{eq:color-chain}, since their row is constant
except for an increase of one at the change. This case contributes
\[
\sum_{\nu=1}^{p-1}\Phi_{p,\nu}
+\ind{d=1}\sum_{\nu=1}^{p-1}w^pz^{p-2\nu+1}
\]
to the value at $\xi=d$.

We now combine the cases for $\xi\ge d$. For $\xi=d$, the
always-upper subcases (a)--(c) and the lower-then-upper case apply.
For $\xi>d$, all four always-upper subcases apply. The common term
$\sum_{\nu=1}^{p-1}\Phi_{p,\nu}$ comes from the lower-then-upper
case when $\xi=d$ and from subcase (d) when $\xi>d$.
In both cases, each fixed $\nu$ gives one internal configuration
and one singleton choice for each of the labels $B,C$.
Thus the weight sums at $\xi=d$ and $\xi>d$ share the polynomial
\begin{equation}\label{eq:Gamma}
\begin{aligned}
\Gamma_p(d)={}&w^pz^{p+1}+\min(d,p+1)w^pz^{p-1}+w^{p+1}z^p
 +\sum_{\nu=1}^{H_p(d)}\binom{H_p(d)}\nu\Psi_{p,\nu}\\
&+\sum_{\nu=1}^{p}\bigl(\min(d-1,p+1-\nu)+\ind{\nu<p}\bigr)\Phi_{p,\nu},
\end{aligned}
\end{equation}
and the additional singleton choices give
\begin{equation}\label{eq:UV}
U_p(d)=\Gamma_p(d)+\ind{d=1}\sum_{\nu=1}^{p-1}w^pz^{p-2\nu+1},
\qquad V_p(d)=\Gamma_p(d)+w^pz^{p-1}.
\end{equation}
Empty sums are zero. The complete boundary weight sum is
\begin{equation}\label{eq:profile-cases}
b_p(\xi;d)=\begin{cases}
N_p,&\xi=\bot,\\
R_p+D_p,&\xi=0,\\
R_p,&0<\xi<d,\\
U_p(d),&\xi=d,\\
V_p(d),&\xi>d.
\end{cases}
\end{equation}
For $\xi\ne\bot$, we require $\xi\le n-a$.
By \eqref{eq:color-chain}, the relative rows of vertices labelled
$A$, $B$, and $C$ are at most $0,d,\xi$, respectively.
Thus all relative rows are at most $\max(d,\xi)$.
Since $a+d\le n$ and $a+\xi\le n$, adding $a$ back shows that
all original boundary rows are at most $n$.
In original row coordinates, the first inequality in
\eqref{eq:interlace} gives $h_0<q_1\le n$, so the singleton
row is at most $n-1$.
The cases of Step~4 give all valid lower boundary configurations,
with each configuration counted once.

\smallskip\noindent\emph{Step 5. Joining the boundaries and extracting the counts.}
We now count the middle labellings compatible with the two boundaries.
Fix $1\le a<b=a+d\le n$ and $p,s\ge0$ with $p+s\le4$.
By Step~3, the first and last $AB$-levels are $p+1$ and $k+1-s$,
with $M=k-1-p-s$ middle levels between them.
We have $M-d+1=k-p-s-d\ge n-p-s\ge0$.
A binomial coefficient with nonnegative upper argument is taken
to be zero if its lower argument is larger.

The lower boundary's last $C$, when present, is at level $p$ and
row $a+\xi$, with $0\le\xi\le n-a$.
For the upper boundary, reflect the coordinates by
$(t,y)\mapsto(k+2-t,n+1-y)$. The reflection reverses the product
order. It sends the last $AB$-level to the first $AB$-level of the
reflected configuration, with the same row difference $d$. The lower
vertex of the reflected $AB$-level carries the original label $B$.
We therefore exchange $A,B$ and apply Step~4.
The first $C$ in the original upper boundary, when present, is at level
$k+2-s=p+M+3$. If its row is $y$, its reflected row is $n+1-y$.
After subtracting $n+1-b$ from the reflected row coordinates, the
last $C$ has relative row $\zeta=b-y$, with $0\le\zeta\le b-1$.
The upper boundary weight is therefore $\overline{b_s(\zeta;d)}$,
where $\overline f(z,w)=f(z^{-1},w)$.
We set $\xi=\bot$ or $\zeta=\bot$ when the corresponding boundary
contains no $C$. The stated bounds on $\xi,\zeta$ apply for every
$p,s$. The boundary weights are zero for values that are not allowed;
for example, $b_0(\xi;d)=0$ when $\xi\ge d$.

At every middle level the rows are $a,b$. Applying
\eqref{eq:color-chain} to the first and last $AB$-levels and the
intervening vertices places middle $A$'s in row $a$ and middle $B$'s
in row $b$. Thus the possible level labels, read from
lower to upper, are $AB,CB,AC$, with weights $1,wz^{-1},wz$.
Once the positions and rows of the middle $C$'s are chosen, all
other labels are determined. The vertices labelled $A$ form a chain,
as do those labelled $B$. We therefore need only check
\eqref{eq:color-chain} between consecutive occurrences of $C$. A middle position $x\in\{1,\ldots,M\}$ has level $p+1+x$.
There are four disjoint cases.

\smallskip\noindent\emph{No middle $C$.}
There is one labelling, with $AB$ at every middle level.
When both boundaries contain $C$, applying \eqref{eq:color-chain}
to the last lower boundary $C$ and the first upper boundary $C$
gives $0\le d-\xi-\zeta\le M+3$.
Since $k\ge2n-1$, $p+s\le4$, and $n\ge4$, we have
$M+3=k+2-p-s\ge2n-3\ge n-1$.
Together with $d\le n-1$ and $\xi,\zeta\ge0$, this gives
$d-\xi-\zeta\le d\le M+3$.
Thus only the lower bound remains, namely $\xi+\zeta\le d$.
For fixed boundary parameters $\xi,\zeta$, the weight sum
for middle labellings with no $C$ is therefore
\[
\eta_d(\xi,\zeta)=
\begin{cases}
1,&\xi=\bot\text{ or }\zeta=\bot,\\
\ind{\xi+\zeta\le d},&\xi,\zeta\ne\bot.
\end{cases}
\]
The case $\xi=\zeta=\bot$ has total $w$-degree zero and therefore
does not contribute to the required coefficients of $w^2,w^3,w^4$.

\smallskip\noindent\emph{All $h>0$ middle $C$'s in row $a$.}
If the lower boundary contains $C$, \eqref{eq:color-chain} between
its last $C$ and the middle $C$'s requires $a+\xi\le a$, hence $\xi=0$. The condition is also
sufficient, since the compared vertices then lie in the same row. Put
$\chi(\xi)=\ind{\xi\in\{\bot,0\}}$.
If the upper boundary contains $C$, \eqref{eq:color-chain} between
a middle $C$ at position $x$ and the first upper boundary $C$ gives
$0\le d-\zeta\le M+2-x$.
For $0\le\zeta\le d$ this excludes exactly the last
$\max(0,d-\zeta-2)$ positions; for $\zeta>d$ it is impossible.
If the upper boundary contains no $C$, all positions are allowed. Set
\[
v_h(\zeta)=
\begin{cases}
\binom Mh,&\zeta=\bot,\\[1mm]
\displaystyle\binom{M-\max(0,d-\zeta-2)}h,&0\le\zeta\le d,\\[1mm]
0,&\zeta>d.
\end{cases}
\]
The bound $M-d+1\ge0$ ensures that the excluded positions lie within
the middle. The middle $C$'s are comparable because they lie in the
same row. Thus the stated boundary comparisons suffice. The weight sum for these middle labellings is
$w^hz^{-h}\chi(\xi)v_h(\zeta)$.

\smallskip\noindent\emph{All $h>0$ middle $C$'s in row $b$.}
By \eqref{eq:color-chain}, we need $\zeta\in\{\bot,0\}$ and,
when $\xi\ne\bot$, $0\le d-\xi\le x+1$.
Reversing the middle positions by $x\mapsto M+1-x$ gives
the same position count $v_h(\xi)$ as in the case with all middle $C$'s in row $a$.
Each middle $C$ is paired with $A$, so the weight sum for
these middle labellings is $w^hz^hv_h(\xi)\chi(\zeta)$.

\smallskip\noindent\emph{Both middle rows contain $C$.}
Write the selected positions as $x_1<\cdots<x_h$.
By \eqref{eq:color-chain}, the rows along $C$ cannot decrease. Thus the first $\ell$ occurrences
lie in row $a$ and the remaining $h-\ell$ in row $b$, for some
$1\le\ell<h$. The middle vertices form a chain if and only if
$x_{\ell+1}-x_\ell\ge d$.
For fixed $\ell$, no placement exists when $M<h+d-1$.
Otherwise, the $M-h$ unselected levels all have labels $AB$.
First place $d-1$ of them between occurrences $\ell$ and $\ell+1$.
The remaining $M-h-d+1$ levels can be distributed among the $h+1$
gaps before, between, and after the selected levels. The number of such distributions is
$\binom{M-d+1}{h}$. The binomial coefficient is also
zero when $M<h+d-1$. Each choice of gap lengths determines exactly
one set of positions, and conversely. There is no $\binom h\ell$
factor: for fixed $\ell$, the row of each occurrence is determined.
The boundary comparisons require $\chi(\xi)\chi(\zeta)=1$.
When present, the last lower boundary $C$ lies in row $a$, and the
first upper boundary $C$ lies in row $b$. Thus the remaining
comparisons along $C$ hold. The weight for fixed
$h,\ell$ is therefore
$\chi(\xi)\chi(\zeta)\binom{M-d+1}{h}w^hz^{h-2\ell}$.

Adding the four cases gives the middle weight. Boundary vertices
are not included, and only terms of $w$-degree at most four are needed:
\begin{equation}\label{eq:central-weight}
\begin{aligned}
W_{M,d}(\xi,\zeta)={}&\eta_d(\xi,\zeta)
 +\sum_{h=1}^{4}w^h\bigl(z^{-h}\chi(\xi)v_h(\zeta)
                         +z^hv_h(\xi)\chi(\zeta)\bigr)\\
&+\chi(\xi)\chi(\zeta)\sum_{h=2}^{4}w^h
  \binom{M-d+1}{h}\sum_{\ell=1}^{h-1}z^{h-2\ell}.
\end{aligned}
\end{equation}
The different powers of $z$ in the sum over $\ell$ record the
different chain sizes and must be kept. For fixed parameters, the sum of the weights of all compatible labellings is
\[
b_p(\xi;d)\,\overline{b_s(\zeta;d)}\,W_{M,d}(\xi,\zeta);
\]
the first and last $AB$-levels have weight $1$.
Every labelling with $A$ below $B$ determines $p,s,a,d$ from its
first and last $AB$-levels. The last lower boundary $C$ and the first upper boundary
$C$, or their absence, determine $\xi,\zeta$.
Conversely, joining any boundaries and middle labelling counted in Steps~4 and~5
gives row data satisfying \eqref{eq:interlace}. The comparisons already checked
ensure, by transitivity, that each label gives a chain. Every internal
boundary level contains $C$, so the chosen $AB$-levels are the first
and last $AB$-levels of the resulting labelling. Thus each such
labelling is counted exactly once. By Step~2, the row data determine
$\Upsilon$ uniquely, so no further factor is needed for the paths.

To select the prescribed chain sizes, put
$e_j=\mu_{j,1}-\mu_{j,2}$ and $c_j=\mu_{j,3}$.
In the order of \cref{tab:six}, their pairs are
$(0,4),(2,4),(2,2),(1,3),(0,2),(3,3)$.
There are $2k+4$ remaining vertices, so specifying
$\#C=c_j$ and $\#A-\#B=\pm e_j$ fixes the other two sizes as
$(2k+4-c_j\pm e_j)/2$.
Our convention places $A$ below $B$. The positive sign assigns
sizes $\mu_{j,1},\mu_{j,2}$ to the lower and upper chains,
respectively; the negative sign reverses the assignments.
For $e_j=0$, both coefficient extractions must still be included,
even though they are equal. The two orders of distinct equal-sized
chains give two semiordered partitions. Both chain sizes are at least
$k-1\ge6$, whereas $c_j\le4$, so the chain $C$ requires no additional
ordering factor. We obtain
\begin{equation}\label{eq:direct-F}
\begin{aligned}
F_j(n,k)={}&\bigl([z^{e_j}w^{c_j}]+[z^{-e_j}w^{c_j}]\bigr)
 \sum_{d=1}^{n-1}\sum_{a=1}^{n-d}
 \sum_{\substack{p,s\ge0\\p+s\le4}}\\
&\quad\sum_{\xi\in\{\bot,0,\ldots,n-a\}}
 \sum_{\zeta\in\{\bot,0,\ldots,a+d-1\}}
 b_p(\xi;d)\,\overline{b_s(\zeta;d)}\,
 W_{k-1-p-s,d}(\xi,\zeta).
\end{aligned}
\end{equation}
If a term with $h$ middle occurrences of $C$ contributes to
\eqref{eq:direct-F}, then $p+s+h\le c_j\le4$, hence
$M-d+1\ge n-p-s\ge4-p-s\ge h$; the upper arguments of the binomial
coefficients in $v_h$ are also at least $M-d+1$. Thus the position
counts are valid even for the smallest allowed parameters $n=4,k=7$.
All factors have nonnegative $w$-degree, so discarding terms of
degree at least five at any stage cannot change the required
coefficients. Step~6 evaluates the row sums and coefficients in
\eqref{eq:direct-F}.
\par\smallskip\noindent\emph{Step 6. Evaluating the six counts.}
We first sum over the row parameters in the four cases of Step~5.
We then extract the coefficients that give the prescribed chain sizes.
We obtain formulas for $d\ge5$ and then compute the corrections at
$d=1,2,3,4$.
We use the boundary polynomials from Step~4 and the middle weights from
Step~5. Throughout, $M=k-1-p-s$, and a bar replaces $z$ by $z^{-1}$.

\smallskip\noindent\emph{6.1. Summing over the row parameters.}
First suppose the middle has no $C$. There are three disjoint cases
in which at least one boundary has no $C$: neither boundary contains
$C$, only the upper boundary contains $C$, or only the lower boundary
contains $C$. Their weights at fixed $a$ are, respectively,
$N_p\overline{N_s}$,
$N_p(d\overline{R_s}+\overline{D_s}+\overline{U_s}+(a-1)\overline{V_s})$,
and $(dR_p+D_p+U_p+(n-a-d)V_p)\overline{N_s}$.
The values of $\xi$ or $\zeta$ other than $\bot$ are grouped as
$0$, $1,\ldots,d-1$, $d$, and $>d$. The extra weight $D$ occurs
only at zero.
If both boundaries have $C$, the condition is $\xi+\zeta\le d$.
There are $\binom{d+2}{2}-2$ pairs satisfying
$0\le\xi,\zeta<d$ and $\xi+\zeta\le d$.
If either parameter is zero, the other has $d$ possible values in this range.
The two remaining pairs are $(d,0),(0,d)$. Finally,
$\sum_{a=1}^{n-d}(a-1)=\sum_{a=1}^{n-d}(n-a-d)
=(n-d)(n-d-1)/2$.
After summing over $a,\xi,\zeta$ and dividing by $n-d$, the weight
for the case with no middle $C$ is
\begin{equation}\label{eq:row-none}
\begin{aligned}
P^0_{p,s}={}&N_p\overline{N_s}+D_p\overline{D_s}
 +N_p\left(d\overline{R_s}+\overline{D_s}+\overline{U_s}
                  +\frac{n-d-1}{2}\overline{V_s}\right)\\
&+\left(dR_p+D_p+U_p+\frac{n-d-1}{2}V_p\right)\overline{N_s}
 +\left(\binom{d+2}{2}-2\right)R_p\overline{R_s}\\
&+d(D_p\overline{R_s}+R_p\overline{D_s})
 +U_p(\overline{R_s}+\overline{D_s})+(R_p+D_p)\overline{U_s}.
\end{aligned}
\end{equation}
The count $\binom{d+2}{2}-2$ excludes $\bot$. It is also valid at
$d=1$, when $(0,0)$ is the only pair with $0\le\xi,\zeta<d$.

If all $h>0$ middle $C$'s lie in row $a$, the middle factor is
$w^hz^{-h}\chi(\xi)v_h(\zeta)$. Summing over $\xi$ gives $N_p+R_p+D_p$.
For the upper boundary, the values $\zeta=\bot$, $\zeta=0$,
$1\le\zeta<d$, $\zeta=d$, and $\zeta>d$ give, respectively,
$\overline{N_s}\binom Mh$,
$(\overline{R_s}+\overline{D_s})v_h(0)$,
$\overline{R_s}\sum_{\xi=1}^{d-1}v_h(\xi)$,
$\overline{U_s}\binom Mh$, and zero. For the sum of the $R_s$ terms, Pascal's
identity gives
\begin{equation}\label{eq:Sh-row}
\sum_{\xi=0}^{d-1}v_h(\xi)
=\binom Mh+\binom{M+1}{h+1}-\binom{M-d+2}{h+1}.
\end{equation}
Indeed, for $d\ge2$ the sum is
$\binom Mh+\sum_{t=M-d+2}^{M}\binom th$; for $d=1$ the last
two terms in \eqref{eq:Sh-row} cancel. The summand is independent
of $a$, so summing over $a$ gives a factor of $n-d$.
After dividing by $n-d$, the weight for this case is
\begin{equation}\label{eq:row-a}
P^a_{h,p,s}=w^hz^{-h}(N_p+R_p+D_p)
 \left[(\overline{N_s}+\overline{U_s})\binom Mh
 +\overline{R_s}\sum_{\xi=0}^{d-1}v_h(\xi)+\overline{D_s}v_h(0)\right].
\end{equation}
The sums over the two boundary parameters can be taken in either
order. Their restrictions differ because the middle $C$'s lie in
row $a$. Exchanging $p,s$ and replacing $z$ by $z^{-1}$ in \eqref{eq:row-a}
gives the corresponding weight, also divided by $n-d$, when all
middle $C$'s lie in row $b$:
\begin{equation}\label{eq:row-b}
P^b_{h,p,s}=w^hz^h
 \left[(N_p+U_p)\binom Mh+R_p\sum_{\xi=0}^{d-1}v_h(\xi)+D_pv_h(0)\right]
 (\overline{N_s}+\overline{R_s}+\overline{D_s}).
\end{equation}
The reflection identity is
\[
P^b_{h,p,s}(z,w)=P^a_{h,s,p}(z^{-1},w).
\]
After summing over $p,s$, the cases with all middle $C$'s in row $a$
or all in row $b$ give the same sum of the two required coefficients.
Indeed, reflection interchanges $p,s$ and replaces $z$ by $z^{-1}$.
It preserves the range of summation and exchanges the two coefficient
extractions. The equality need not hold at one fixed ordered pair $(p,s)$.
When both middle rows contain $C$, both boundary sums are restricted
by $\chi$, and the position count from Step~5 yields, for $h\ge2$,
\begin{equation}\label{eq:row-both}
P^{ab}_{h,p,s}=(N_p+R_p+D_p)
 (\overline{N_s}+\overline{R_s}+\overline{D_s})
 w^h\binom{M-d+1}{h}\sum_{\ell=1}^{h-1}z^{h-2\ell}.
\end{equation}
Each of the four weights includes both boundaries and has been divided
by $n-d$. With the parameter ranges of \eqref{eq:direct-F}, we therefore have
\begin{equation}\label{eq:row-bridge}
\begin{aligned}
\sum_{a=1}^{n-d}\sum_\xi\sum_\zeta
 b_p(\xi;d)\overline{b_s(\zeta;d)}W_{M,d}(\xi,\zeta)
&\equiv(n-d)\Bigl[P^0_{p,s}
 +\sum_{h=1}^4(P^a_{h,p,s}+P^b_{h,p,s})\\
&\qquad +\sum_{h=2}^4P^{ab}_{h,p,s}\Bigr]\pmod{w^5}.
\end{aligned}
\end{equation}

\smallskip\noindent\emph{6.2. The generic polynomials $B_j$.}
For a fixed gap $d$, \eqref{eq:row-bridge} gives the weight sum
after summing over the row parameters and dividing by $n-d$.
We apply the two coefficient extractions in \eqref{eq:direct-F}
and sum over the boundary lengths and middle cases.
Here $h=0$ when the middle has no $C$. Each internal boundary
level contains one $C$, so only $p+s+h\le c_j$ can contribute.
The remaining $c_j-p-s-h$ occurrences must lie at the two singleton
levels; a term requiring more than two has zero coefficient.
In each summand, $M=k-1-p-s$.

For $d\ge5$ and $1\le p\le4$, the substitutions in \eqref{eq:Gamma} are
$H_p(d)=0$, $\min(d,p+1)=p+1$, and
$\min(d-1,p+1-\nu)=p+1-\nu$ for $1\le\nu\le p$.
We call the resulting boundary polynomials generic and write
$U_p^{\mathrm g}=\Gamma_p^{\mathrm g}$ and
$V_p^{\mathrm g}=\Gamma_p^{\mathrm g}+w^pz^{p-1}$;
$N_p,R_p,D_p$ and all singleton polynomials are unchanged.
Also $v_h(0)=\binom{M-d+2}h$ in this range. In every weight, we
use \eqref{eq:Sh-row} for the sum $\sum_{\xi=0}^{d-1}v_h(\xi)$
and $\binom{M-d+2}h$ for the separately occurring factor $v_h(0)$.
With these substitutions, \eqref{eq:direct-F} and
\eqref{eq:row-bridge} give the following finite formula for $B_j$:
\begin{equation}\label{eq:B-generic-evaluation}
\begin{aligned}
B_j={}&\bigl([z^{e_j}w^{c_j}]+[z^{-e_j}w^{c_j}]\bigr)\Biggl[
 \sum_{\substack{p,s\ge0\\p+s\le c_j}}P^0_{p,s}
 +2\sum_{\substack{h\ge1,\ p,s\ge0\\h+p+s\le c_j}}P^a_{h,p,s}
&+\sum_{\substack{h\ge2,\ p,s\ge0\\h+p+s\le c_j}}P^{ab}_{h,p,s}\Biggr].
\end{aligned}
\end{equation}
All indices are integers, and each ordered pair $(p,s)$ is used once.
As in \eqref{eq:direct-F}, both coefficient extractions are retained
when $e_j=0$. The factor $2$ accounts for the equal contributions
from the cases with all middle $C$'s in row $a$ or all in row $b$,
after summing over $p,s$ and adding the two required coefficients.
Expanding the resulting binomial expressions gives
\eqref{eq:B1}--\eqref{eq:B6}; the full evaluations are given in
Subsections~\ref{app:B1}--\ref{app:B6}.

Equations~\eqref{eq:row-bridge} and~\eqref{eq:B-generic-evaluation} show that each gap $d$ with
$5\le d\le n-1$ contributes $(n-d)B_j(n,k,d)$ to $F_j$. The finite sums define $B_j$ as a polynomial
in $n,k,d$. At an admissible gap $d\le4$, its value need not equal
the actual contribution divided by $n-d$; Step~6.3 computes the
difference. The polynomial definition also applies when $n=4,5$,
although no admissible gap then satisfies $d\ge5$.

\smallskip\noindent\emph{6.3. Small-gap corrections.}
The generic formulas give the actual contributions for $d\ge5$.
We now compare them with the actual contributions at
$d=1,2,3,4$. We will show that the correction at $d=4$ is zero,
so only $d=1,2,3$ require correction terms.
For $1\le r\le\min(4,n-1)$, the correction $\varepsilon_{j,r}$
is the actual contribution at gap $r$, including its row factor $n-r$,
minus $(n-r)B_j(n,k,r)$. Both terms use the same gap. They differ only in the boundary
polynomials and the factor $v_h(0)$ that occurs separately in the weights.
Write $\Delta_rU_p=U_p(r)-U_p^{\mathrm g}$ and similarly for $V_p$.
The polynomials $N_p,R_p,D_p$ are unchanged. For $p\ge1$, use
$H_p(r)=\max(0,p+1-r)$ from Step~4. Subtraction in \eqref{eq:Gamma} gives
\begin{equation}\label{eq:delta-Gamma}
\Delta_rV_p=-H_p(r)w^pz^{p-1}
 +\sum_{\nu=1}^{H_p(r)}
 \left[\binom{H_p(r)}\nu\Psi_{p,\nu}
                  -(H_p(r)+1-\nu)\Phi_{p,\nu}\right].
\end{equation}
The remaining polynomial difference is
$\Delta_rU_p=\Delta_rV_p+\ind{r=1}
\sum_{\nu=1}^{p-1}w^pz^{p-2\nu+1}$.
The nonzero values of these differences modulo $w^5$ are listed in
\cref{tab:profile-changes-one,tab:profile-changes-two-three}
in Appendix~\ref{app:small-gap-corrections}.

The formula for $\sum_{\xi=0}^{r-1}v_h(\xi)$ in \eqref{eq:Sh-row}
is valid for every $r\ge1$, so this entire sum has no correction.
For the factor $v_h(0)$, the difference between its actual value
and the generic formula is
\begin{equation}\label{eq:delta-v0}
\binom{M-\max(0,r-2)}h-\binom{M-r+2}h
=-\ind{r=1}\binom M{h-1}.
\end{equation}
For each of the four middle cases, let $\Delta_rP$ be the difference
between the weights obtained from the actual polynomials and $v_h(0)$
and from their generic formulas, both evaluated at $d=r$.
Using \eqref{eq:delta-Gamma} and \eqref{eq:delta-v0}, together with the formula
for $\Delta_rU_p$, in \eqref{eq:row-none}, \eqref{eq:row-a},
\eqref{eq:row-b}, and \eqref{eq:row-both} gives
\begin{align}
\Delta_rP^0_{p,s}
={}&\Delta_rU_p(\overline{N_s}+\overline{R_s}+\overline{D_s})
 +(N_p+R_p+D_p)\overline{\Delta_rU_s}+\frac{n-r-1}{2}
 (\Delta_rV_p\overline{N_s}+N_p\overline{\Delta_rV_s}),
\label{eq:delta-none}\\
\Delta_rP^a_{h,p,s}
={}&w^hz^{-h}(N_p+R_p+D_p)
 \left[\binom Mh\overline{\Delta_rU_s}
             -\ind{r=1}\binom M{h-1}\overline{D_s}\right],
\label{eq:delta-a}\\
\Delta_rP^b_{h,p,s}
={}&\overline{\Delta_rP^a_{h,s,p}},\qquad
\Delta_rP^{ab}_{h,p,s}=0.
\label{eq:delta-b-both}
\end{align}
The correction is therefore obtained by applying the coefficient
extraction and index ranges of \eqref{eq:B-generic-evaluation} to
these weight differences and multiplying by $n-r$. Equivalently,
\begin{equation}\label{eq:correction-definition}
\varepsilon_{j,r}
 =\bigl(\text{actual gap-$r$ contribution}\bigr)-(n-r)B_j(n,k,r).
\end{equation}
For $r=1,2,3$, Appendix~\ref{app:small-gap-corrections} evaluates
these coefficient extractions. By \eqref{eq:delta-b-both}, the
cases with all middle $C$'s in row $a$ or all in row $b$ give equal
sums after summing over $p,s$ and adding the two required coefficients.
The case with middle $C$'s in both rows gives zero.
Thus we add the sum for no middle $C$ and twice the sum for row $a$,
then multiply by $n-r$. This gives the entries of
\cref{tab:corrections}.

Finally, at $r=4$ all polynomials with $p\le3$ agree with their generic formulas.
For $p=4$, $H_4(4)=1$ and the entire difference in \eqref{eq:delta-Gamma}
is $-w^4z^3+\Psi_{4,1}-\Phi_{4,1}=0$.
The extra term in $\Delta_rU_p$ is zero at $r=4$, and the difference
in \eqref{eq:delta-v0} is also zero;
hence all four weight differences vanish before coefficient extraction.
Thus $\varepsilon_{j,4}=0$ for every $j$. When $n=4$, gap four is
outside the sum, so this identity adds no extra contribution.
The exceptional gaps $1,2,3$ are present for all $n\ge4$.
Summing \eqref{eq:correction-definition} over the admissible gaps
proves \eqref{eq:gap-count}, and \eqref{eq:factor-yang} then proves
\eqref{eq:counts-prop}.
\end{proof}
\Needspace{6\baselineskip}
\begin{proof}[Proof of \cref{thm:main}]
By \cref{cor:WW-poset,lem:force,prop:counts} and the signs in
\cref{tab:six},
\begin{equation}\label{eq:signed-six}
[s_{\rho}]X_{\inc(\bchain m\times\bchain n)}
=(n-2)!\,(F_1-F_2-F_3-F_4+F_5+F_6).
\end{equation}
Substituting \eqref{eq:gap-count} and combining the corrections in
\cref{tab:corrections} gives
\begin{equation}\label{eq:main-substitution}
\begin{aligned}
[s_{\rho}]X_{\inc(\bchain m\times\bchain n)}
={}&(n-2)!\Biggl\{\sum_{d=1}^{n-1}(n-d)(B_1-B_2-B_3-B_4+B_5+B_6)\\
&\quad +(n-1)(k^2+5k+6n-18)+(n-2)(2k+3n-5)+(n-3)(n-2)\Biggr\}.
\end{aligned}
\end{equation}
Here each $B_j$ in the sum is evaluated at $(n,k,d)$.
Their signed combination has degree at most four in $d$.
Substitute \eqref{eq:B1}--\eqref{eq:B6} and use
\[
\sum_{d=1}^{n-1}(n-d)d^\ell
=n\sum_{d=1}^{n-1}d^\ell-\sum_{d=1}^{n-1}d^{\ell+1}
\qquad(0\le\ell\le4).
\]
The standard power-sum formulas through degree five then yield
\begin{equation}\label{eq:main-evaluation}
\begin{aligned}
[s_{\rho}]X_{\inc(\bchain m\times\bchain n)}
={}&\frac{(n-2)!}{360}\Bigl\{(-20n^3+20n)k^3
 +(15n^4+30n^3-195n^2+510n-360)k^2\\
&\quad +(-6n^5-15n^4-100n^3+195n^2+2446n-3240)k\\
&\quad +n^6+3n^5+25n^4+45n^3+4294n^2-15168n+12240\Bigr\}.
\end{aligned}
\end{equation}
Equation~\eqref{eq:main-evaluation} proves the coefficient identity
in \eqref{eq:main-coeff}.

To determine the sign of the coefficient, set $u=n-4\ge0$ and $t=k-(2n-1)\ge0$.
Substituting $n=u+4$ and $k=2u+7+t$ into \eqref{eq:main-evaluation} gives
\begin{equation}\label{eq:negative}
\begin{aligned}
[s_{\rho}]X_{\inc(\bchain m\times\bchain n)}
={}&-\frac{(n-2)!}{360}\Bigl\{210240+408900u+300744u^2+110355u^3+22065u^4+2385u^5\\
&\quad +111u^6+(122640+174794u+95535u^2+25480u^3+3405u^4+186u^5)t\\
&\quad +(20880+22710u+9075u^2+1590u^3+105u^4)t^2\\
&\quad +(1200+940u+240u^2+20u^3)t^3\Bigr\}.
\end{aligned}
\end{equation}
Since $u,t\ge0$, every term inside the braces is nonnegative, and the
constant term is $210240>0$. The factor outside the braces is strictly negative.
Thus the Schur coefficient in \eqref{eq:main-coeff} is negative, and
$\bchain m\times\bchain n$ is not Schur positive.
\end{proof}

\begin{corollary}\label{cor:range}
The product $\bchain m\times\bchain n$ is not Schur positive if either
$n=2$ and $m\ge8$, or $n\ge3$ and $m\ge n+5$.
\end{corollary}

\begin{proof}[Proof of \cref{cor:range}]
For $n\ge4$ and $k=m-n\ge5$, either $5\le k\le2n-2$, when \eqref{eq:old-range} applies, or $k\ge2n-1$, when \cref{thm:main} applies. The cases $n=2$, $m\ge8$, and $n=3$, $m\ge8$, are the cited results of \citet{WZ26}.
\end{proof}

\Needspace{6\baselineskip}
\appendix
\section{\texorpdfstring{Evaluation of the polynomials $B_j$}{Evaluation of the polynomials Bj}}\label{app:counts}\label{app:other-counts}

We evaluate \eqref{eq:B-generic-evaluation} using the generic boundary
polynomials defined in Step~6.2. Throughout this appendix, $U_p,V_p$
denote these polynomials. We set $M=k-1-p-s$ separately for each ordered
pair $(p,s)$. All contributions have been divided by $n-d$, so they add
directly to $B_j$. For $5\le d\le n-1$, the polynomials give the gap-$d$
contribution to $F_j$, divided by $n-d$. Evaluating the polynomials at
smaller gaps gives the comparison values used in Step~6.3.

The boundary inputs are
\[
N_0=z+z^{-1},\quad R_0=w,\quad D_0=U_0=V_0=0,
\]
and, for $1\le p\le4$,
\[
N_p=0,\quad R_p=w^pz^{-p-1}+w^{p+1}z^{-p},\quad
D_p=w^pz^{1-p},\quad V_p=U_p+w^pz^{p-1}.
\]
The generic substitution in \eqref{eq:Gamma} gives
\begin{align*}
U_1={}&w(z^2+2+z^{-2})+w^2(z+z^{-1}),\\*
U_2={}&w^2(z^3+3z+3z^{-1}+z^{-3})+w^3(z^2+3+z^{-2}),\\*
U_3={}&w^3(z^4+4z^2+4+3z^{-2}+z^{-4})
       +w^4(z^3+4z+3z^{-1}+z^{-3}),\\*
U_4\equiv{}&w^4(z^5+5z^3+5z+4z^{-1}+3z^{-3}+z^{-5})\pmod{w^5}.
\end{align*}
Terms of degree at least five cannot affect the required coefficients.
A bar replaces $z$ by $z^{-1}$ in every term, including the additional
term in $V_p$.

For target $(e_j,c_j)$, first take the coefficient of $w^{c_j}$.
Then add the coefficients of $z^{e_j}$ and $z^{-e_j}$ in the resulting
Laurent polynomial. Thus the coefficient column for $e_j=0$ already
counts the constant term twice. The index bounds
are $h+p+s\le c_j$, with $h=0$ in the no-middle class. The remaining
$c_j-h-p-s$ occurrences of $C$ must occupy the two singleton levels.
If more than two occurrences remain, the coefficient is zero. Each table contains all admissible cases,
including zeros, with the following conventions:
\begin{itemize}[itemsep=1pt,topsep=2pt]
\item In the no-middle block, only $p\le s$ is listed.
Equation~\eqref{eq:row-none} gives
$P^0_{s,p}(z,w)=P^0_{p,s}(z^{-1},w)$, so $(p,s)$ and $(s,p)$
give equal sums of the two required coefficients. The indicated
multiplicity is therefore $2$ when $p<s$ and $1$ when $p=s$.
\item In the single-$a$ block, all ordered triples are listed. Substitute
the displayed $M$ and add the coefficients. Equations~\eqref{eq:row-a}
and~\eqref{eq:row-b} give
$P^b_{h,p,s}(z,w)=P^a_{h,s,p}(z^{-1},w)$.
Thus the complete single-$b$ sum equals the single-$a$ sum,
and the single-$a$ total is used twice.
\item Each two-row entry sums all ordered pairs with the indicated
$h,p+s$ and all $1\le\ell<h$, and includes its position factor
$\binom{k-p-s-d}{h}$. Add these entries directly.
\end{itemize}

For example, at $(p,s)=(0,2)$ the first summand of
\eqref{eq:row-none} involving $N_0$ has $w$-degree at most three.
The degree-four coefficient therefore comes from the term with $R_0=w$:
\begin{align*}
[w^4]P^0_{0,2}
&=[w^3]\left[\left(\binom{d+2}{2}-2\right)\overline R_2
                         +d\overline D_2+\overline U_2\right]=\left(\binom{d+2}{2}-1\right)z^2+3+z^{-2}.
\end{align*}
The contributions to $B_1$ and $B_2$ are $6$ and
$\binom{d+2}{2}$, respectively, before multiplication by $2$
to include $(p,s)=(2,0)$.
The other no-middle entries follow by the same substitution.

For single-row entries we use \eqref{eq:Sh-row} and, in particular,
\[
\sum_{\xi=0}^{d-1}v_1(\xi)=dM-\binom{d-1}{2},
\qquad v_h(0)=\binom{M-d+2}{h}.
\]
Each subsection illustrates this extraction for one ordered triple.
For the two-row entries, the identity
$N_p+R_p+D_p=w^pz^{-p}(z+z^{-1}+w)$, valid also at $p=0$, reduces
\eqref{eq:row-both} to
\[
P^{ab}_{h,p,s}=w^{h+p+s}z^{s-p}(z+z^{-1}+w)^2
 \binom{k-p-s-d}{h}\sum_{\ell=1}^{h-1}z^{h-2\ell}.
\]
The last sum is $1$, $z+z^{-1}$, or $z^2+1+z^{-2}$ for
$h=2,3,4$, respectively. Substituting these sums into \eqref{eq:row-both} and expanding gives the grouped two-row entries. We now add the three blocks, with the
single-row factor of two, and expand their binomial coefficients.

\subsection{\texorpdfstring{Evaluation of $B_1$: $(e_1,c_1)=(0,4)$}{Evaluation of B1: (e,c)=(0,4)}}\label{app:B1}
The target is $2[z^0w^4]$. The no-middle class has $p+s\le4$;
the single-row classes have $1\le h\le4$ and $p+s\le4-h$;
the two-row class has $2\le h\le4$ and $p+s\le4-h$.
All boundary lengths are nonnegative, and $M=k-1-p-s$.

To illustrate the single-row entries, take $(h,p,s)=(1,0,2)$, so
$M=k-3$. The position factors in \eqref{eq:row-a} are
\[
\binom M1=M,\qquad
\sum_{\xi=0}^{d-1}v_1(\xi)=dM-\frac{(d-1)(d-2)}2,\qquad
v_1(0)=M-d+2.
\]
Substituting the boundary inputs gives
\begin{align*}
P^a_{1,0,2}={}&wz^{-1}(z+z^{-1}+w)\Bigl\{
 M\bigl[w^2(z^{-3}+3z^{-1}+3z+z^3)\\*
&\qquad +w^3(z^{-2}+3+z^2)\bigr]
 +(w^2z^3+w^3z^2)\left(dM-\frac{(d-1)(d-2)}2\right)\\*
&\qquad +w^2z(M-d+2)\Bigr\}.
\end{align*}
The $w^4$ coefficient is
\begin{align*}
[w^4]P^a_{1,0,2}={}&M(2z^2+7+7z^{-2}+2z^{-4})+\left(dM-\frac{(d-1)(d-2)}2\right)(2z^2+1)+(M-d+2).
\end{align*}
Taking the constant term twice and substituting $M=k-3$ gives
\begin{align*}
2[z^0w^4]P^a_{1,0,2}
 &=2\left[7M+dM-\frac{(d-1)(d-2)}2+(M-d+2)\right]\\*
 &=2(d+8)(k-3)-(d-1)(d-2)-2d+4.
\end{align*}

\Cref{tab:app-B1-data} records all cases. Its single-row entries already add the
two target extractions, and its two-row entries also include the
position factor and the sums over the ordered boundary lengths and
$\ell$ for each listed $(h,p+s)$.

\begin{table}[!htbp]\centering\small
\setlength{\tabcolsep}{3pt}
\renewcommand{\arraystretch}{1.12}
\caption{Coefficient data for $B_1$.}\label{tab:app-B1-data}
\smallskip
\begin{minipage}[t]{0.36\linewidth}\centering\vspace{0pt}
\emph{No middle $C$}\par\smallskip
\begin{tabular}{@{}ccl@{}}\toprule
$(p,s)$ & Mult. & Coefficient\\\midrule
$(0,0)$ & 1 & $0$\\*
$(0,1)$ & 2 & $0$\\*
$(0,2)$ & 2 & $6$\\*
$(0,3)$ & 2 & $-7d+7n+15$\\*
$(0,4)$ & 2 & $-9d+9n+9$\\*
$(1,1)$ & 1 & $d^{2}+3d+2$\\*
$(1,2)$ & 2 & $d^{2}+5d+16$\\*
$(1,3)$ & 2 & $2d+16$\\*
$(2,2)$ & 1 & $d^{2}+3d+16$\\\bottomrule
\end{tabular}
\par\medskip
\emph{Middle $C$ in both rows}\par\smallskip
\begin{tabular}{@{}cl@{}}\toprule
$(h,p+s)$ & Contribution\\\midrule
$(2,0)$ & $2\binom{k-d}{2}$\\*
$(2,1)$ & $8\binom{k-d-1}{2}$\\*
$(2,2)$ & $8\binom{k-d-2}{2}$\\*
$(3,0)$ & $8\binom{k-d}{3}$\\*
$(3,1)$ & $12\binom{k-d-1}{3}$\\*
$(4,0)$ & $8\binom{k-d}{4}$\\\bottomrule
\end{tabular}
\end{minipage}\hfill
\begin{minipage}[t]{0.62\linewidth}\centering\vspace{0pt}
\emph{Middle $C$ only in row $a$}\par\smallskip
\begin{tabular}{@{}ccl@{}}\toprule
$(h,p,s)$ & $M$ & Coefficient\\\midrule
$(1,0,0)$ & $k-1$ & $0$\\*
$(1,0,1)$ & $k-2$ & $(2d+2)M-2\binom{d-1}{2}$\\*
$(1,0,2)$ & $k-3$ & $(2d+16)M-2\binom{d-1}{2}-2d+4$\\*
$(1,0,3)$ & $k-4$ & $16M-2d+4$\\*
$(1,1,0)$ & $k-2$ & $0$\\*
$(1,1,1)$ & $k-3$ & $(4d+4)M-4\binom{d-1}{2}$\\*
$(1,1,2)$ & $k-4$ & $(2d+10)M-2\binom{d-1}{2}-2d+4$\\*
$(1,2,0)$ & $k-3$ & $0$\\*
$(1,2,1)$ & $k-4$ & $(2d+2)M-2\binom{d-1}{2}$\\*
$(1,3,0)$ & $k-4$ & $0$\\*
$(2,0,0)$ & $k-1$ & $0$\\*
$(2,0,1)$ & $k-2$ & $8\binom{M}{2}+4\binom{M+1}{3}-4\binom{M-d+2}{3}$\\*
$(2,0,2)$ & $k-3$ & $10\binom{M}{2}+2\binom{M+1}{3}-2\binom{M-d+2}{3}+2\binom{M-d+2}{2}$\\*
$(2,1,0)$ & $k-2$ & $0$\\*
$(2,1,1)$ & $k-3$ & $4\binom{M}{2}+2\binom{M+1}{3}-2\binom{M-d+2}{3}$\\*
$(2,2,0)$ & $k-3$ & $0$\\*
$(3,0,0)$ & $k-1$ & $0$\\*
$(3,0,1)$ & $k-2$ & $4\binom{M}{3}+2\binom{M+1}{4}-2\binom{M-d+2}{4}$\\*
$(3,1,0)$ & $k-2$ & $0$\\*
$(4,0,0)$ & $k-1$ & $0$\\\bottomrule
\end{tabular}
\end{minipage}
\end{table}

\smallskip\noindent\emph{No middle $C$.}
Using the listed multiplicities, the no-middle-$C$ entries in
\cref{tab:app-B1-data} sum to
\[
4d^2-12d+32n+142.
\]

\smallskip\noindent\emph{All middle $C$ in row $a$.}
For each entry in the single-$a$ block of \cref{tab:app-B1-data}, substitute its listed
value of $M$. Summing these entries gives
\begin{align*}
&-\frac{d^4}{12}+\frac{d^3(2k+5)}6
 -\frac{d^2(6k^2+30k-1)}{12}\qquad+\frac{d(2k^3+15k^2-k-89)}6
 +\frac{k^3+12k^2+23k-288}3.
\end{align*}
By \eqref{eq:row-a} and \eqref{eq:row-b}, the total for row $b$
is the same.

\smallskip\noindent\emph{Middle $C$ in both rows.}
Adding the grouped two-row entries in \cref{tab:app-B1-data}, with their position factors
already included, gives
\[
\frac{(k-d)^4}{3}+\frac{4(k-d)^3}{3}-\frac{10(k-d)^2}{3}+\frac{31d}{3}-\frac{31k}{3}+20.
\]

Finally, \eqref{eq:B-generic-evaluation} adds the no-middle-$C$ total,
twice the single-$a$ total, and the two-row total. This gives
\begin{align*}
6B_1={}&d^4-2d^3(2k-1)+d^2(6k^2-6k+5)\\
&-2d(2k^3-3k^2-19k+94)
 +2(k^4+6k^3+14k^2+15k+96n-90),
\end{align*}
which is \eqref{eq:B1}.

\subsection{\texorpdfstring{Evaluation of $B_2$: $(e_2,c_2)=(2,4)$}{Evaluation of B2: (e,c)=(2,4)}}\label{app:B2}
For every weight, add its $z^2w^4$ and $z^{-2}w^4$ coefficients.
The ranges are $p,s\ge0$ with $p+s\le4$ when there is no middle $C$;
$1\le h\le4$ and $h+p+s\le4$ for a single middle row; and
$2\le h\le4$ and $h+p+s\le4$ when both middle rows contain $C$.
We use \eqref{eq:row-none}, \eqref{eq:row-a}, and \eqref{eq:row-both},
with $M=k-1-p-s$.

For example, take $(h,p,s)=(2,0,1)$ in the
single-$a$ class, so $M=k-2$. Substitution in \eqref{eq:row-a} gives
\begin{align*}
P^a_{2,0,1}={}&w^2z^{-2}(z+z^{-1}+w)\Biggl\{
 \bigl[w(z^{-2}+2+z^2)+w^2(z^{-1}+z)\bigr]\binom M2\\*
&\qquad +(wz^2+w^2z)\sum_{\xi=0}^{d-1}v_2(\xi)
               +wv_2(0)\Biggr\}.
\end{align*}
Its complete $w^4$ coefficient is
\begin{align*}
[w^4]P^a_{2,0,1}
 ={}&(2+4z^{-2}+2z^{-4})\binom M2 +(2+z^{-2})\sum_{\xi=0}^{d-1}v_2(\xi)+z^{-2}v_2(0).
\end{align*}
The $z^2$ coefficient is zero, while the $z^{-2}$ coefficient is
\[
4\binom M2+\sum_{\xi=0}^{d-1}v_2(\xi)+v_2(0).
\]
For $M=k-2$, the three position factors are
\begin{align*}
\binom M2&=\binom{k-2}{2}, \quad \sum_{\xi=0}^{d-1}v_2(\xi)=\binom{k-2}{2}+\binom{k-1}{3}-\binom{k-d}{3}, \quad v_2(0)=\binom{k-d}{2}.
\end{align*}
Thus the sum of the two target coefficients is
\[
5\binom{k-2}{2}+\binom{k-1}{3}-\binom{k-d}{3}+\binom{k-d}{2}.
\]
\Cref{tab:app-B2-data} lists the coefficients for all allowed parameters.
For the case with no middle $C$, it lists only $p\le s$, with
multiplicity $2$ when $p<s$ and $1$ when $p=s$, as follows from \eqref{eq:row-none}. For the case with middle $C$ only in row $a$, it lists
all twenty ordered triples and their values of $M$. The two-row entries are grouped by $(h,p+s)$;
each already includes all ordered boundary pairs, all divisions
between rows $a,b$, and the position factor $\binom{k-d-p-s}{h}$.

\begin{table}[!htbp]\centering\small
\setlength{\tabcolsep}{3pt}
\renewcommand{\arraystretch}{1.12}
\caption{Coefficient data for $B_2$.}\label{tab:app-B2-data}
\smallskip
\begin{minipage}[t]{0.36\linewidth}\centering\vspace{0pt}
\emph{No middle $C$}\par\smallskip
\begin{tabular}{@{}ccl@{}}\toprule
$(p,s)$ & Mult. & Coefficient\\\midrule
$(0,0)$ & 1 & $0$\\*
$(0,1)$ & 2 & $0$\\*
$(0,2)$ & 2 & $\frac{d^{2}+3d+2}{2}$\\*
$(0,3)$ & 2 & $\frac{-5d+9n+23}{2}$\\*
$(0,4)$ & 2 & $-9d+9n+9$\\*
$(1,1)$ & 1 & $2$\\*
$(1,2)$ & 2 & $\frac{d^{2}+5d+24}{2}$\\*
$(1,3)$ & 2 & $\frac{d^{2}+3d+30}{2}$\\*
$(2,2)$ & 1 & $2d+14$\\\bottomrule
\end{tabular}
\par\medskip
\emph{Middle $C$ in both rows}\par\smallskip
\begin{tabular}{@{}cl@{}}\toprule
$(h,p+s)$ & Contribution\\\midrule
$(2,0)$ & $0$\\*
$(2,1)$ & $4\binom{k-d-1}{2}$\\*
$(2,2)$ & $6\binom{k-d-2}{2}$\\*
$(3,0)$ & $4\binom{k-d}{3}$\\*
$(3,1)$ & $8\binom{k-d-1}{3}$\\*
$(4,0)$ & $6\binom{k-d}{4}$\\\bottomrule
\end{tabular}
\end{minipage}\hfill
\begin{minipage}[t]{0.62\linewidth}\centering\vspace{0pt}
\emph{Middle $C$ only in row $a$}\par\smallskip
\begin{tabular}{@{}ccl@{}}\toprule
$(h,p,s)$ & $M$ & Coefficient\\\midrule
$(1,0,0)$ & $k-1$ & $0$\\*
$(1,0,1)$ & $k-2$ & $M$\\*
$(1,0,2)$ & $k-3$ & $(2d+9)M-2\binom{d-1}{2}$\\*
$(1,0,3)$ & $k-4$ & $(d+13)M-\binom{d-1}{2}-d+2$\\*
$(1,1,0)$ & $k-2$ & $dM-\binom{d-1}{2}$\\*
$(1,1,1)$ & $k-3$ & $(d+5)M-\binom{d-1}{2}-d+2$\\*
$(1,1,2)$ & $k-4$ & $(d+8)M-\binom{d-1}{2}-d+2$\\*
$(1,2,0)$ & $k-3$ & $(d+1)M-\binom{d-1}{2}$\\*
$(1,2,1)$ & $k-4$ & $(d+4)M-\binom{d-1}{2}-d+2$\\*
$(1,3,0)$ & $k-4$ & $M$\\*
$(2,0,0)$ & $k-1$ & $\binom{M}{2}+\binom{M+1}{3}-\binom{M-d+2}{3}$\\*
$(2,0,1)$ & $k-2$ & $5\binom{M}{2}+\binom{M+1}{3}-\binom{M-d+2}{3}+\binom{M-d+2}{2}$\\*
$(2,0,2)$ & $k-3$ & $8\binom{M}{2}+\binom{M+1}{3}-\binom{M-d+2}{3}+\binom{M-d+2}{2}$\\*
$(2,1,0)$ & $k-2$ & $2\binom{M}{2}+\binom{M+1}{3}-\binom{M-d+2}{3}$\\*
$(2,1,1)$ & $k-3$ & $4\binom{M}{2}+\binom{M+1}{3}-\binom{M-d+2}{3}+\binom{M-d+2}{2}$\\*
$(2,2,0)$ & $k-3$ & $\binom{M}{2}$\\*
$(3,0,0)$ & $k-1$ & $2\binom{M}{3}+\binom{M+1}{4}-\binom{M-d+2}{4}$\\*
$(3,0,1)$ & $k-2$ & $4\binom{M}{3}+\binom{M+1}{4}-\binom{M-d+2}{4}+\binom{M-d+2}{3}$\\*
$(3,1,0)$ & $k-2$ & $\binom{M}{3}$\\*
$(4,0,0)$ & $k-1$ & $\binom{M}{4}$\\\bottomrule
\end{tabular}
\end{minipage}
\end{table}

\smallskip\noindent\emph{No middle $C$.}
Using the listed multiplicities, the no-middle-$C$ entries in
\cref{tab:app-B2-data} sum to
\[
3d^2-10d+27n+113.
\]

\smallskip\noindent\emph{All middle $C$ in row $a$.}
For each entry in the single-$a$ block of \cref{tab:app-B2-data}, substitute its listed
value of $M$. Summing these entries gives
\begin{align*}
&-\frac{d^4}{12}+\frac{d^3(k+1)}3
 -\frac{d^2(6k^2+12k+5)}{12}\\
&\qquad+\frac{d(2k^3+6k^2+5k-65)}6
 +\frac{k^4+14k^3+59k^2+94k-1848}{24}.
\end{align*}
By \eqref{eq:row-a} and \eqref{eq:row-b}, the total for row $b$
is the same.

\smallskip\noindent\emph{Middle $C$ in both rows.}
Adding the grouped two-row entries in \cref{tab:app-B2-data}, with their position factors
already included, gives
\[
\frac{(k-d)^4}{4}+\frac{(k-d)^3}{2}-\frac{9(k-d)^2}{4}+\frac{13d}{2}-\frac{13k}{2}+14.
\]

Finally, \eqref{eq:B-generic-evaluation} adds the no-middle-$C$ total,
twice the single-$a$ total, and the two-row total. This gives
\begin{align*}
12B_2={}&d^4-2d^3(2k-1)+d^2(6k^2-6k-1)\\
&-2d(2k^3-3k^2-37k+151)
 +4(k^4+5k^3+8k^2+4k+81n-81),
\end{align*}
which is \eqref{eq:B2}.

\subsection{\texorpdfstring{Evaluation of $B_3$: $(e_3,c_3)=(2,2)$}{Evaluation of B3: (e,c)=(2,2)}}\label{app:B3}
The target contribution of each weight is $[z^2w^2]P+[z^{-2}w^2]P$.
\Cref{tab:app-B3-data} records the no-middle, single-$a$, and two-row contributions.
For an example with middle $C$ only in row $a$, take $(h,p,s)=(1,0,1)$,
so $M=k-2$. Substitution in \eqref{eq:row-a} gives
\begin{align*}
P^a_{1,0,1}={}&wz^{-1}(z+z^{-1}+w)\Biggl\{
 \bigl[w(z^{-2}+2+z^2)+w^2(z^{-1}+z)\bigr]M\\*
&\qquad +(wz^2+w^2z)\left(dM-\binom{d-1}{2}\right)
       +w(M-d+2)\Biggr\},\\*
[w^2]P^a_{1,0,1}={}&M(z^2+3+3z^{-2}+z^{-4})+\left(dM-\binom{d-1}{2}\right)(z^2+1)
       +(M-d+2)(1+z^{-2}).
\end{align*}
Thus the two target coefficients and their sum are
\begin{align*}
[z^2w^2]P^a_{1,0,1}&=(d+1)M-\binom{d-1}{2},\\*
[z^{-2}w^2]P^a_{1,0,1}&=4M-d+2,\\*
[z^2w^2]P^a_{1,0,1}+[z^{-2}w^2]P^a_{1,0,1}
 &=(d+5)(k-2)-\binom{d-1}{2}-d+2.
\end{align*}
The other entries follow by the same substitution and extraction.

\begin{table}[!htbp]\centering\small
\setlength{\tabcolsep}{3pt}
\renewcommand{\arraystretch}{1.12}
\caption{Coefficient data for $B_3$.}\label{tab:app-B3-data}
\smallskip
\begin{minipage}[t]{0.36\linewidth}\centering\vspace{0pt}
\emph{No middle $C$}\par\smallskip
\begin{tabular}{@{}ccl@{}}\toprule
$(p,s)$ & Mult. & Coefficient\\\midrule
$(0,0)$ & 1 & $0$\\*
$(0,1)$ & 2 & $\frac{d^{2}+3d+2n+4}{2}$\\*
$(0,2)$ & 2 & $\frac{-7d+9n+9}{2}$\\*
$(1,1)$ & 1 & $2d+8$\\\bottomrule
\end{tabular}
\par\medskip
\emph{Middle $C$ in both rows}\par\smallskip
\begin{tabular}{@{}cl@{}}\toprule
$(h,p+s)$ & Contribution\\\midrule

$(2,0)$ & $2\binom{k-d}{2}$\\\bottomrule
\end{tabular}
\end{minipage}\hfill
\begin{minipage}[t]{0.62\linewidth}\centering\vspace{0pt}
\emph{Middle $C$ only in row $a$}\par\smallskip
\begin{tabular}{@{}ccl@{}}\toprule
$(h,p,s)$ & $M$ & Coefficient\\\midrule
$(1,0,0)$ & $k-1$ & $(d+1)M-\binom{d-1}{2}$\\*
$(1,0,1)$ & $k-2$ & $(d+5)M-\binom{d-1}{2}-d+2$\\*
$(1,1,0)$ & $k-2$ & $2M$\\*
$(2,0,0)$ & $k-1$ & $2\binom{M}{2}$\\\bottomrule
\end{tabular}
\end{minipage}
\end{table}

\smallskip\noindent\emph{No middle $C$.}
Using the listed multiplicities, the no-middle-$C$ entries in
\cref{tab:app-B3-data} sum to
\begin{equation}\label{eq:B3-none}
\sum_{\substack{p,s\ge0\\p+s\le2}}\bigl([z^2w^2]+[z^{-2}w^2]\bigr)P^0_{p,s}=d^2-2d+11n+21.
\end{equation}

\smallskip\noindent\emph{All middle $C$ in row $a$.}
For each entry in the single-$a$ block of \cref{tab:app-B3-data}, substitute its listed
value of $M$. Summing these entries gives
\begin{equation}\label{eq:B3-a}
\sum_{\substack{h\ge1,\ p,s\ge0\\h+p+s\le2}}\bigl([z^2w^2]+[z^{-2}w^2]\bigr)P^a_{h,p,s}=-d^2+2dk-d+k^2+5k-13.
\end{equation}
By \eqref{eq:row-a} and \eqref{eq:row-b}, the total for row $b$
is the same.

\smallskip\noindent\emph{Middle $C$ in both rows.}
The only case is $(h,p,s)=(2,0,0)$, whose entry in \cref{tab:app-B3-data} gives
\[
2\binom{k-d}{2}=(k-d)(k-d-1).
\]

Finally, \eqref{eq:B-generic-evaluation} adds the no-middle-$C$ total,
twice the single-$a$ total, and the two-row total. This gives
\begin{equation}\label{eq:B3-evaluation}
B_3=(2k-3)d+3k^2+9k+11n-5,
\end{equation}
which is \eqref{eq:B3}.

\subsection{\texorpdfstring{Evaluation of $B_4$: $(e_4,c_4)=(1,3)$}{Evaluation of B4: (e,c)=(1,3)}}\label{app:B4}
The target is $([z^1]+[z^{-1}])[w^3]$, with $h+p+s\le3$:
$h=0$ in the no-middle class, $1\le h\le3$ in a single-row class,
and $2\le h\le3$ in the two-row class.
For example, in the case with middle $C$ only in row $a$, take $(h,p,s)=(1,1,1)$,
set $M=k-3$ and substitute the boundary inputs in \eqref{eq:row-a}:
\begin{align*}
P^a_{1,1,1}={}&wz^{-1}\bigl[w(z^{-2}+1)+w^2z^{-1}\bigr]\Bigl\{ M\bigl[w(z^{-2}+2+z^2)+w^2(z^{-1}+z)\bigr]\\*
&\quad +(wz^2+w^2z)\left(dM-\frac{(d-1)(d-2)}2\right)
 +w(M-d+2)\Bigr\}.
\end{align*}
Taking $[w^3]$ gives
\begin{align*}
[w^3]P^a_{1,1,1}={}&M(z+3z^{-1}+3z^{-3}+z^{-5})\\*
&+\left(dM-\frac{(d-1)(d-2)}2\right)(z+z^{-1})
 +(M-d+2)(z^{-1}+z^{-3}).
\end{align*}
Hence the two target coefficients are
\begin{align*}
[zw^3]P^a_{1,1,1}&=(d+1)M-\frac{(d-1)(d-2)}2,\\*
[z^{-1}w^3]P^a_{1,1,1}&=(d+4)M-\frac{(d-1)(d-2)}2-d+2.
\end{align*}
Their sum at $M=k-3$ is $(2d+5)(k-3)-(d-1)(d-2)-d+2$.

The complete case totals are listed in \cref{tab:app-B4-data}; the two-row entries include
the position factors and the sums over ordered boundaries and $\ell$.

\begin{table}[!htbp]\centering\small
\setlength{\tabcolsep}{3pt}
\renewcommand{\arraystretch}{1.12}
\caption{Coefficient data for $B_4$.}\label{tab:app-B4-data}
\smallskip
\begin{minipage}[t]{0.36\linewidth}\centering\vspace{0pt}
\emph{No middle $C$}\par\smallskip
\begin{tabular}{@{}ccl@{}}\toprule
$(p,s)$ & Mult. & Coefficient\\\midrule
$(0,0)$ & 1 & $0$\\*
$(0,1)$ & 2 & $\frac{d^{2}+3d+2}{2}$\\*
$(0,2)$ & 2 & $-2d+4n+10$\\*
$(0,3)$ & 2 & $-8d+8n+8$\\*
$(1,1)$ & 1 & $d^{2}+5d+10$\\*
$(1,2)$ & 2 & $\frac{d^{2}+5d+28}{2}$\\\bottomrule
\end{tabular}
\par\medskip
\emph{Middle $C$ in both rows}\par\smallskip
\begin{tabular}{@{}cl@{}}\toprule
$(h,p+s)$ & Contribution\\\midrule
$(2,0)$ & $4\binom{k-d}{2}$\\*
$(2,1)$ & $6\binom{k-d-1}{2}$\\*
$(3,0)$ & $6\binom{k-d}{3}$\\\bottomrule
\end{tabular}
\end{minipage}\hfill
\begin{minipage}[t]{0.62\linewidth}\centering\vspace{0pt}
\emph{Middle $C$ only in row $a$}\par\smallskip
\begin{tabular}{@{}ccl@{}}\toprule
$(h,p,s)$ & $M$ & Coefficient\\\midrule
$(1,0,0)$ & $k-1$ & $dM-\binom{d-1}{2}$\\*
$(1,0,1)$ & $k-2$ & $(3d+7)M-3\binom{d-1}{2}-d+2$\\*
$(1,0,2)$ & $k-3$ & $(d+12)M-\binom{d-1}{2}-2d+4$\\*
$(1,1,0)$ & $k-2$ & $(d+1)M-\binom{d-1}{2}$\\*
$(1,1,1)$ & $k-3$ & $(2d+5)M-2\binom{d-1}{2}-d+2$\\*
$(1,2,0)$ & $k-3$ & $M$\\*
$(2,0,0)$ & $k-1$ & $2\binom{M}{2}+\binom{M+1}{3}-\binom{M-d+2}{3}$\\*
$(2,0,1)$ & $k-2$ & $6\binom{M}{2}+2\binom{M+1}{3}-2\binom{M-d+2}{3}+\binom{M-d+2}{2}$\\*
$(2,1,0)$ & $k-2$ & $\binom{M}{2}$\\*
$(3,0,0)$ & $k-1$ & $\binom{M}{3}$\\\bottomrule
\end{tabular}
\end{minipage}
\end{table}

\smallskip\noindent\emph{No middle $C$.}
Using the listed multiplicities, the no-middle-$C$ entries in
\cref{tab:app-B4-data} sum to
\[
3d^2-7d+24n+76.
\]

\smallskip\noindent\emph{All middle $C$ in row $a$.}
For each entry in the single-$a$ block of \cref{tab:app-B4-data}, substitute its listed
value of $M$. Summing these entries gives
\begin{align*}
&\frac{d^3}{2}-\frac{d^2(3k+5)}2
 +\frac{d(3k^2+10k-18)}2+\frac{k^3+15k^2+62k-300}{6}.
\end{align*}
By \eqref{eq:row-a} and \eqref{eq:row-b}, the total for row $b$
is the same.

\smallskip\noindent\emph{Middle $C$ in both rows.}
Adding the grouped two-row entries in \cref{tab:app-B4-data}, with their position factors
already included, gives
\[
(k-d)^3+2(k-d)^2+9d-9k+6.
\]

Finally, \eqref{eq:B-generic-evaluation} adds the no-middle-$C$ total,
twice the single-$a$ total, and the two-row total. This gives
\[
3B_4=6d(3k-8)+4k^3+21k^2+35k+72n-54,
\]
which is \eqref{eq:B4}.

\subsection{\texorpdfstring{Evaluation of $B_5$: $(e_5,c_5)=(0,2)$}{Evaluation of B5: (e,c)=(0,2)}}\label{app:B5}
The target is $2[z^0w^2]$, so the constant term is counted twice.
For an example with middle $C$ only in row $a$, take $(h,p,s)=(1,0,1)$.
Then $M=k-2$, $\sum_{\xi=0}^{d-1}v_1(\xi)=dM-\binom{d-1}{2}$,
and $v_1(0)=M-d+2$. Substitution in \eqref{eq:row-a} gives
\begin{align*}
P^a_{1,0,1}={}&wz^{-1}(z+z^{-1}+w)\Bigl\{
 \bigl[w(z^{-2}+2+z^2)+w^2(z^{-1}+z)\bigr]M\\*
&\qquad +(wz^2+w^2z)\left(dM-\binom{d-1}{2}\right)
       +w(M-d+2)\Bigr\}.
\end{align*}
Taking the $w$-term of the braces and the $w^0$-term of $z+z^{-1}+w$ yields
\begin{align*}
[w^2]P^a_{1,0,1}
={}&M(z^2+3+3z^{-2}+z^{-4})+\left(dM-\binom{d-1}{2}\right)(z^2+1)
       +(M-d+2)(1+z^{-2}),\\*
2[z^0w^2]P^a_{1,0,1}
={}&2\left(3M+dM-\binom{d-1}{2}+M-d+2\right)
=2(d+4)(k-2)-2\binom{d-1}{2}-2d+4.
\end{align*}
The remaining extractions are listed in \cref{tab:app-B5-data}. In its no-middle-$C$
part, \eqref{eq:row-none} gives equal contributions for $(p,s)$
and $(s,p)$. Thus the multiplicities are $1,2,2,1$ for
$(0,0),(0,1),(0,2),(1,1)$.
The four single-$a$ triples are ordered and each is used once;
the two-row entry includes its position factor.

\begin{table}[!htbp]\centering\small
\setlength{\tabcolsep}{3pt}
\renewcommand{\arraystretch}{1.12}
\caption{Coefficient data for $B_5$.}\label{tab:app-B5-data}
\smallskip
\begin{minipage}[t]{0.36\linewidth}\centering\vspace{0pt}
\emph{No middle $C$}\par\smallskip
\begin{tabular}{@{}ccl@{}}\toprule
$(p,s)$ & Mult. & Coefficient\\\midrule
$(0,0)$ & 1 & $d^{2}+3d-2$\\*
$(0,1)$ & 2 & $2d+2n+6$\\*
$(0,2)$ & 2 & $-7d+7n+7$\\*
$(1,1)$ & 1 & $d^{2}+3d+12$\\\bottomrule
\end{tabular}
\par\medskip
\emph{Middle $C$ in both rows}\par\smallskip
\begin{tabular}{@{}cl@{}}\toprule
$(h,p+s)$ & Contribution\\\midrule

$(2,0)$ & $4\binom{k-d}{2}$\\\bottomrule
\end{tabular}
\end{minipage}\hfill
\begin{minipage}[t]{0.62\linewidth}\centering\vspace{0pt}
\emph{Middle $C$ only in row $a$}\par\smallskip
\begin{tabular}{@{}ccl@{}}\toprule
$(h,p,s)$ & $M$ & Coefficient\\\midrule
$(1,0,0)$ & $k-1$ & $(2d+2)M-2\binom{d-1}{2}$\\*
$(1,0,1)$ & $k-2$ & $(2d+8)M-2\binom{d-1}{2}-2d+4$\\*
$(1,1,0)$ & $k-2$ & $2M$\\*
$(2,0,0)$ & $k-1$ & $2\binom{M}{2}$\\\bottomrule
\end{tabular}
\end{minipage}
\end{table}

\smallskip\noindent\emph{No middle $C$.}
Using the listed multiplicities, the no-middle-$C$ entries in
\cref{tab:app-B5-data} sum to
\[
\sum_{\substack{p,s\ge0\\p+s\le2}}2[z^0w^2]P^0_{p,s}=2d^2-4d+18n+36.
\]

\smallskip\noindent\emph{All middle $C$ in row $a$.}
For each entry in the single-$a$ block of \cref{tab:app-B5-data}, substitute its listed
value of $M$. Summing these entries gives
\[
\sum_{\substack{h\ge1,\ p,s\ge0\\h+p+s\le2}}2[z^0w^2]P^a_{h,p,s}=-2d^2+2d(2k-1)+k^2+9k-20.
\]
By \eqref{eq:row-a} and \eqref{eq:row-b}, the total for row $b$
is the same.

\smallskip\noindent\emph{Middle $C$ in both rows.}
The only case is $(h,p,s)=(2,0,0)$, whose entry in \cref{tab:app-B5-data} gives
\[
4\binom{k-d}{2}=2k^2-4kd+2d^2-2k+2d.
\]

Finally, \eqref{eq:B-generic-evaluation} adds the no-middle-$C$ total,
twice the single-$a$ total, and the two-row total. This gives
\[
B_5=2\bigl((2k-3)d+2k^2+8k+9n-2\bigr),
\]
which is \eqref{eq:B5}.

\subsection{\texorpdfstring{Evaluation of $B_6$: $(e_6,c_6)=(3,3)$}{Evaluation of B6: (e,c)=(3,3)}}\label{app:B6}
The target is $([z^3]+[z^{-3}])[w^3]$, with $h+p+s\le3$:
$h=0$ in the no-middle class, $1\le h\le3$ in a single-row class,
and $2\le h\le3$ in the two-row class.
For example, in the case with middle $C$ only in row $a$, take $(h,p,s)=(1,0,2)$,
set $M=k-3$ and substitute the boundary inputs in \eqref{eq:row-a}:
\begin{align*}
P^a_{1,0,2}={}&wz^{-1}(z+z^{-1}+w)\Bigl\{
 M\bigl[w^2(z^{-3}+3z^{-1}+3z+z^3) +w^3(z^{-2}+3+z^2)\bigr]\\*
&\quad +(w^2z^3+w^3z^2)\left(dM-\frac{(d-1)(d-2)}2\right)
 +w^2z(M-d+2)\Bigr\}.
\end{align*}
Taking $[w^3]$ gives
\begin{align*}
[w^3]P^a_{1,0,2}={}&M(z^3+4z+6z^{-1}+4z^{-3}+z^{-5})\\*
&+\left(dM-\frac{(d-1)(d-2)}2\right)(z^3+z)
 +(M-d+2)(z+z^{-1}).
\end{align*}
Hence the two target coefficients are
\[
[z^3w^3]P^a_{1,0,2}=(d+1)M-\frac{(d-1)(d-2)}2,
\qquad [z^{-3}w^3]P^a_{1,0,2}=4M.
\]
Their sum at $M=k-3$ is $(d+5)(k-3)-(d-1)(d-2)/2$.

The complete case totals are listed in \cref{tab:app-B6-data}; the two-row entries include
the position factors and the sums over ordered boundaries and $\ell$.

\begin{table}[!htbp]\centering\small
\setlength{\tabcolsep}{3pt}
\renewcommand{\arraystretch}{1.12}
\caption{Coefficient data for $B_6$.}\label{tab:app-B6-data}
\smallskip
\begin{minipage}[t]{0.36\linewidth}\centering\vspace{0pt}
\emph{No middle $C$}\par\smallskip
\begin{tabular}{@{}ccl@{}}\toprule
$(p,s)$ & Mult. & Coefficient\\\midrule
$(0,0)$ & 1 & $0$\\*
$(0,1)$ & 2 & $0$\\*
$(0,2)$ & 2 & $\frac{d^{2}+3d+2n+4}{2}$\\*
$(0,3)$ & 2 & $-4d+5n+5$\\*
$(1,1)$ & 1 & $4$\\*
$(1,2)$ & 2 & $d+8$\\\bottomrule
\end{tabular}
\par\medskip
\emph{Middle $C$ in both rows}\par\smallskip
\begin{tabular}{@{}cl@{}}\toprule
$(h,p+s)$ & Contribution\\\midrule
$(2,0)$ & $0$\\*
$(2,1)$ & $2\binom{k-d-1}{2}$\\*
$(3,0)$ & $2\binom{k-d}{3}$\\\bottomrule
\end{tabular}
\end{minipage}\hfill
\begin{minipage}[t]{0.62\linewidth}\centering\vspace{0pt}
\emph{Middle $C$ only in row $a$}\par\smallskip
\begin{tabular}{@{}ccl@{}}\toprule
$(h,p,s)$ & $M$ & Coefficient\\\midrule
$(1,0,0)$ & $k-1$ & $0$\\*
$(1,0,1)$ & $k-2$ & $2M$\\*
$(1,0,2)$ & $k-3$ & $(d+5)M-\binom{d-1}{2}$\\*
$(1,1,0)$ & $k-2$ & $(d+1)M-\binom{d-1}{2}$\\*
$(1,1,1)$ & $k-3$ & $4M-d+2$\\*
$(1,2,0)$ & $k-3$ & $2M$\\*
$(2,0,0)$ & $k-1$ & $2\binom{M}{2}+\binom{M+1}{3}-\binom{M-d+2}{3}$\\*
$(2,0,1)$ & $k-2$ & $3\binom{M}{2}+\binom{M-d+2}{2}$\\*
$(2,1,0)$ & $k-2$ & $2\binom{M}{2}$\\*
$(3,0,0)$ & $k-1$ & $2\binom{M}{3}$\\\bottomrule
\end{tabular}
\end{minipage}
\end{table}

\smallskip\noindent\emph{No middle $C$.}
Using the listed multiplicities, the no-middle-$C$ entries in
\cref{tab:app-B6-data} sum to
\[
d^2-3d+12n+34.
\]

\smallskip\noindent\emph{All middle $C$ in row $a$.}
For each entry in the single-$a$ block of \cref{tab:app-B6-data}, substitute its listed
value of $M$. Summing these entries gives
\begin{align*}
&\frac{d^3}{6}-\frac{d^2(k+1)}2
 +\frac{d(3k^2+6k-16)}6\\
&\qquad+\frac{2k^3+9k^2+13k-144}{6}.
\end{align*}
By \eqref{eq:row-a} and \eqref{eq:row-b}, the total for row $b$
is the same.

\smallskip\noindent\emph{Middle $C$ in both rows.}
Adding the grouped two-row entries in \cref{tab:app-B6-data}, with their position factors
already included, gives
\[
\frac{(k-d)^3}{3}+\frac{7d}{3}-\frac{7k}{3}+2.
\]

Finally, \eqref{eq:B-generic-evaluation} adds the no-middle-$C$ total,
twice the single-$a$ total, and the two-row total. This gives
\[
B_6=2d(k-3)+k^3+3k^2+2k+12n-12,
\]
which is \eqref{eq:B6}.

\section{Evaluation of the small-gap corrections}\label{app:small-gap-corrections}

We evaluate the coefficient extractions in Step~6.3 using
\eqref{eq:delta-Gamma}--\eqref{eq:delta-b-both}.
We keep the notation of that step: $\Delta_rU_p$ and $\Delta_rV_p$
are the actual boundary polynomials at gap $r$ minus their generic
formulas. Throughout, $M=k-1-p-s$, and a bar replaces $z$ by $z^{-1}$.

Substituting $H_p(r)=0,1,2,3,4$ gives the nonzero changes modulo $w^5$
listed in \cref{tab:profile-changes-one,tab:profile-changes-two-three}.
All omitted pairs, including $p=0,1$, have zero change.

\begin{table}[!htbp]
\centering\small
\caption{Changes in the boundary polynomials at $r=1$, modulo $w^5$.}\label{tab:profile-changes-one}
\smallskip
\begin{tabular}{@{}cccc@{}}
\toprule
$p$ & $2$ & $3$ & $4$\\
\midrule
$\Delta_1V_p$ & $w^2z^{-1}$ & $w^3(3+2z^{-2})+w^4z^{-1}$
 & $w^4(6z+7z^{-1}+3z^{-3})$\\
$\Delta_1U_p-\Delta_1V_p$ & $w^2z$ & $w^3(z^2+1)$ & $w^4(z^3+z+z^{-1})$\\
\bottomrule
\end{tabular}
\end{table}

\begin{table}[!htbp]
\centering\small
\caption{Changes in the boundary polynomials at $r=2,3$, modulo $w^5$.}\label{tab:profile-changes-two-three}
\smallskip
\begin{tabular}{@{}cccc@{}}
\toprule
$(r,p)$ & $(2,3)$ & $(2,4)$ & $(3,4)$\\
\midrule
$\Delta_rV_p$ & $w^3$ & $w^4(3z+2z^{-1})$ & $w^4z$\\
$\Delta_rU_p-\Delta_rV_p$ & $0$ & $0$ & $0$\\
\bottomrule
\end{tabular}
\end{table}

For example, when $H_p(r)=2$, \eqref{eq:delta-Gamma} reduces to
$-2w^pz^{p-1}+2\Psi_{p,1}+\Psi_{p,2}
-2\Phi_{p,1}-\Phi_{p,2}=w^pz^{p-3}$.
Terms of $w$-degree at least five cannot affect any required coefficient.

We compute the required coefficients of the weight differences using
$N_p+R_p+D_p=w^pz^{-p}(z+z^{-1}+w)$ and
$N_p=\ind{p=0}(z+z^{-1})$.
\Cref{tab:normalized-corrections} records the sums of the two required coefficients over
all allowed indices in each middle case, before multiplication by $n-r$.
By \eqref{eq:delta-b-both}, the correction for middle $C$'s in row
$b$ equals that for row $a$ after summing over $p,s$ and adding the
two required coefficients. The correction for middle $C$'s in both
rows is zero.

\begin{table}[!htbp]
\centering\small
\caption{Coefficient corrections by gap and middle case, before multiplication by $n-r$.}\label{tab:normalized-corrections}
\smallskip
\begin{tabular}{@{}crrrrr@{}}
\toprule
& \multicolumn{2}{c}{$r=1$} & \multicolumn{2}{c}{$r=2$} & $r=3$\\
$j$ & No middle $C$ & Row $a$ & No middle $C$ & Row $a$ & No middle $C$\\
\midrule
$1$ & $28n+56$ & $k^2+7k-50$ & $10n-2$ & $2k-8$ & $2n-4$\\
$2$ & $17n+32$ & $(k^2+5k)/2-26$ & $5n-3$ & $k-4$ & $n-2$\\
$3$ & $n+2$ & $-1$ & $0$ & $0$ & $0$\\
$4$ & $8n+16$ & $2k-11$ & $2n-2$ & $0$ & $0$\\
$5$ & $2n+4$ & $-2$ & $0$ & $0$ & $0$\\
$6$ & $2n+4$ & $-2$ & $0$ & $0$ & $0$\\
\bottomrule
\end{tabular}
\end{table}

All single-row corrections at gap three and all corrections at gap four are zero.
Only a few choices of the indices contribute to these corrections.
At $r=1$, the correction for no middle $C$ uses only the unordered pairs
$(0,2)$, $(0,3)$, $(0,4)$, $(1,2)$, $(1,3)$, $(2,2)$; reflected pairs are counted
separately and $(2,2)$ once. For example, for $j=1$ their contributions
are respectively $0,n+8,13n+4,2,12,4$, giving
$2[(n+8)+(13n+4)+2+12]+4=28n+56$.
For $j=1$, the nonzero contributions with all middle $C$'s in row $a$, at
$(h,p,s)=(1,0,2),(1,0,3),(1,1,2),(2,0,2)$, are
$2k-8,12k-50,2k-10,k^2-9k+18$, whose sum is $k^2+7k-50$.
Thus the entries for $j=1$ and $r=1$ in \cref{tab:normalized-corrections} give
\[
\varepsilon_{1,1}=(n-1)\bigl[(28n+56)+2(k^2+7k-50)\bigr]
=2(n-1)(k^2+7k+14n-22).
\]
For $j=3$ at gap one, only $(p,s)=(0,2),(2,0)$ contribute in the
case with no middle $C$, each by $(n+2)/2$. In the case with all
middle $C$'s in row $a$, only $(h,p,s)=(1,0,1)$ contributes, since
$\Delta_1P^a_{1,0,1}=-w^2(1+z^{-2})-w^3z^{-1}$.
This gives $\varepsilon_{3,1}=(n-1)[(n+2)-1-1]=n(n-1)$.

At $r=2$, only the polynomials for boundary lengths $3,4$ change.
In the case with no middle $C$, the unordered pairs $(0,3),(0,4),(1,3)$
give the respective vectors
$(2,0,0,n-1,0,0)$,
$(5(n-1),5(n-1)/2,0,0,0,0)$, and $(2,1,0,0,0,0)$,
where the entries are the contributions for $j=1,\ldots,6$ in order.
Each pair has a distinct reflected pair. For middle $C$'s in row $a$,
the only weight difference that contributes to a required coefficient is
$\Delta_2P^a_{1,0,3}\equiv(k-4)w^4(1+z^{-2})\pmod{w^5}$,
which gives $2(k-4),k-4$ for $j=1,2$, respectively.
At $r=3$, only $(0,4),(4,0)$ remain, with
$\Delta_3P^0_{0,4}\equiv\frac{n-2}{2}w^4(1+z^{-2})\pmod{w^5}$
and its reflection. Together they give $2(n-2),n-2$ for $j=1,2$,
respectively, and zero for the other four values of $j$.
A middle $C$ would raise the $w$-degree above four.
The remaining entries for gap one follow from the same coefficient
extractions. For each gap, add the entry for no middle $C$ and twice
the entry for middle $C$'s in row $a$, then multiply by $n-r$.
The result agrees with
\eqref{eq:correction-definition} and gives all entries of
\cref{tab:corrections}.

\setlength{\bibsep}{2pt}

\section*{Declarations}

\noindent\textbf{Use of artificial intelligence.}
The author used ChatGPT (OpenAI) to assist with algebraic
simplifications, mathematical and computational checks,
\LaTeX{} typesetting, and English-language editing.
The author assumes responsibility for all content of this article.

\end{document}